\documentclass[10pt,a4paper]{article}
\usepackage[utf8]{inputenc}
\usepackage{amsmath}
\usepackage{amsfonts}
\usepackage{amssymb}
\usepackage{amsthm}
\usepackage{mathrsfs}
\usepackage{enumerate}
\usepackage{tikz-cd}
\usepackage{stmaryrd}
\usepackage{pgf}

\newcommand{\C}{\mathbb{C}}

\newcommand{\Z}{\mathbb{Z}}
\newcommand{\K}{\mathbb{K}}
\newcommand{\N}{\mathbb{N}}

\newcommand{\oh}{\mathcal{O}}
\newcommand{\h}{\mathcal{H}}

\newcommand{\Frac}{\mathrm{Frac}}
\newcommand{\PI}{{\rm PI}\text{-}{\rm deg}}
\newcommand{\End}{\mathrm{End}}

\newcommand{\CSpec}{\mathrm{C.Spec}}

\DeclareMathOperator{\Ima}{Im}
\DeclareMathOperator{\Id}{Id}
\newcommand{\isomarrow}{\xrightarrow{\,\smash{\raisebox{-0.65ex}{\ensuremath{\scriptstyle\sim}}}\,}}


\theoremstyle{plain}
\newtheorem{theorem}{Theorem}[section]
\newtheorem{lemma}[theorem]{Lemma}
\newtheorem{proposition}[theorem]{Proposition}
\newtheorem{corollary}[theorem]{Corollary}
\newtheorem{hypothesis}[theorem]{Hypothesis}

\theoremstyle{definition}
\newtheorem{definition}[theorem]{Definition}
\newtheorem{remark}[theorem]{Remark}

\theoremstyle{remark}

\usepackage[affil-it]{authblk}

\title{A Deleting Derivations Algorithm for Quantum Nilpotent Algebras at Roots of Unity}

\author[1]{St\'{e}phane Launois\thanks{Partially supported by EPSRC grant EP/R009279/1.}}
\affil[1]{School of Mathematics, Statistics and Actuarial Science, University of Kent, Canterbury, Kent, CT2 7FS, UK.}

\author[2]{Samuel A.\ Lopes\thanks{Partially supported by CMUP, member of LASI, which is financed by national funds through FCT -- Funda\c c\~ao para a Ci\^encia e a Tecnologia, I.P., under the projects with reference UIDB/00144/2020 and UIDP/00144/2020.}}

\affil[2]{CMUP, Departamento de Matem\'atica, Faculdade de Ci\^encias, Universidade do Porto, Rua do Campo Alegre s/n, 4169--007 Porto, Portugal.}

\author[1]{Alexandra Rogers}

\begin{document}
\maketitle

\begin{abstract}
This paper extends an algorithm and canonical embedding in \cite{Cauchon} to a large class of quantum algebras. It applies to iterated Ore extensions over a field satisfying some suitable assumptions which cover those of Cauchon's original setting but also allows for roots of unity. The extended algorithm constructs a quantum affine space $A'$ from the original quantum algebra $A$ via a series of change of variables within the division ring of fractions $\Frac(A)$. The canonical embedding takes a completely prime ideal $P\lhd A$ to a completely prime ideal $Q\lhd A'$ such that when $A$ is a PI algebra, $\PI(A/P) = \PI(A'/Q)$. When the quantum parameter is a root of unity,  combining our construction with results from \cite{BellLaunoisRogers} allows us to state an explicit formula for the PI degree of completely prime quotient algebras. This paper ends with a method to construct a maximum dimensional irreducible representation of $A/P$ given a suitable irreducible representation of $A'/Q$ when $A$ is PI.
\end{abstract}

In \cite{Cauchon}, Cauchon developed an algorithm for a large class of quantum algebras originally named CGL extensions after Cauchon, Goodearl and Letzter who first proved a set of unifying results on the algebras in this class, and later renamed quantum nilpotent algebras \cite{LaunoisLenaganRigal2}. These are algebras which can be written in the form of an iterated Ore extension over a field with certain conditions placed on the defining automorphisms and skew-derivations. In essence, the algorithm takes a prime quotient of a suitable CGL extension and, via a series of changes of indeterminates within its division ring of fractions, obtains a prime quotient of a quantum affine space which shares the same division ring of fractions as the original algebra. Properties from the resulting quotiented quantum affine space can then be pulled back to the original quotient algebra. When the prime ideal is invariant under a rational action of a torus, the algorithm combined with Goodearl and Letzter's $\h$-stratification theory was used to prove the \emph{quantum Gel'fand Kirillov Conjecture} on prime quotients of a large class of generic quantum algebras, such as quantum euclidean spaces, quantum symplectic spaces, quantum matrices, and quantum Weyl algebras \cite[Th\'{e}or\`{e}me 6.1.1]{Cauchon}. Cauchon's algorithm has also been applied by many authors to gain a better understanding of the structure of quantum algebras in the generic setting. Such works include the aforementioned \cite{Cauchon-SpectrePremiers} and \cite{GoodearlLetzter} as well as \cite{Casteels,GeigerYakimov,GoodearlLaunoisLenagan-Tauvel,GoodearlYakimov,LaunoisLenaganNolan}. 

Cauchon's algorithm cannot be applied to such algebras at roots of unity, however, as the homomorphism defining the algorithm is not compatible with roots of unity. Several authors have recreated the main results of \cite{Cauchon} for quantum algebras at roots of unity satisfying slightly different conditions (\cite{Haynal} and \cite{LeroyMatczuk}). However, the explicit change of variables in the original algorithm, that was essential in the applications named above, was not recreated in either paper. This paper addresses this by extending Cauchon's deleting derivations algorithm, including the canonical embedding to track completely prime ideals, to include the root of unity case by utilising an adapted version, provided in \cite[Section 3]{Haynal}, of the homomorphism which lies at the heart of the procedure.

The main motivation for this paper is the fact that many quantum algebras become \emph{polynomial identity (PI) algebras} when taken at roots of unity and in this case the \emph{PI degree} is a useful invariant for deducing various properties of the algebra (see, for example,  \cite{BrownYakimov}, \cite{CekenPalmieriWangZhang1}, and \cite{CekenPalmieriWangZhang2}). The PI degree also plays a role in the representation theory of prime affine PI algebras, giving an upper bound on the dimension of irreducible representations \cite[Theorem I.13.5.]{BrownGoodearl}.

The extended algorithm defined in Section \ref{sectionDDA} of this paper takes a suitable iterated Ore extension $A$ and a completely prime ideal $P\in \CSpec(A)$ and constructs a quantum affine space $A'$ and a canonical embedding $\psi:\CSpec(A) \rightarrow \CSpec(A')$ such that $\Frac(A/P)=\Frac(A'/\psi(P))$. Furthermore, if $A$ is PI then, so too is $A'$ and thus $\PI(A/P)=\PI(A'/\psi(P))$. When $A'/\psi(P)$ is itself a quantum affine space then the PI degree is determined by the rank and values of the invariant factors of its commutation matrix (see for instance \cite[Lemma 2.4]{BellLaunoisRogers}). With this in mind, we focus on those ideals $P_w$ which give rise to quantum affine spaces $A'/\psi(P_w)$. For such ideals, the PI degree of $A/P_w$ is then computed in Theorem \ref{thmPIdegCauchonIdeal}. We end with a method to pass certain maximum-dimensional irreducible representations of any $A'/\psi(P)$ back through the deleting derivations algorithm to obtain maximum irreducible representations of $A/P$ (Section \ref{sectionIrredRepDDA}) and illustrate our result by constructing an explicit irreducible representation of maximal dimension for a completely prime quotient of $U_q^+(\mathfrak{s0}_5)$ at roots of unity.\\

\noindent {\bf Acknowledgement:} The material presented in this article is part of the PhD thesis of the third-named author \cite{Thesis}. She would like to thank the School of Mathematics, Statistics and Actuarial Science at the University of Kent for their financial support.

\section{Quantum Nilpotent Algebras at roots of unity}\label{section1}
Take $\K$ to be an arbitrary field and $R$ a $\K$-algebra. One can form an Ore extension $R[x, \sigma, \delta]$ in one indeterminate $x$ by defining a $\K$-automorphism $\sigma$ and a (left) $\sigma$-derivation $\delta$ (also called a \textit{skew-derivation}) such that for any $a, b \in R$ we have $\delta(ab)=\sigma(a)\delta(b)+\delta(a)b$. These maps define the commutation rules for $x$ with elements of $R$: given $r\in R$ we have $xr=\sigma(r)x +\delta(r)$.  Ore extensions satisfy a universal mapping property (\cite[Proposition 2.4]{GoodearlWarfield}) and thus are unique up to isomorphism (\cite[Corollary 2.5]{GoodearlWarfield}).  Iterating this process with $N$ indeterminates gives an iterated Ore extension $R[X_1, \sigma_1, \delta_1]\ldots [X_N, \sigma_N, \delta_N]$. Note that this is a noetherian domain by the Skew Hilbert Basis Theorem \cite[Theorem I.1.13]{BrownGoodearl}. The exact form of $R$ and the properties of the pairs $(\sigma_i, \delta_i)$ are what will define quantum nilpotent algebras at roots of unity and distinguish them from CGL extensions.

If all the skew-derivations in the iterated Ore extension above are the $0$ map, all the automorphisms act by scalar multiplication on each indeterminate, i.e. $\sigma_i(X_j)=\lambda_{i,j} X_j$ for all $1\leq j<i\leq N$ and some $\lambda_{i,j} \in \K^*$, and $R$ is a field $\mathbb{K}$ then the iterated Ore extension would describe a quantum affine space, which we can write as follows:
\[ \K[X_1, \sigma_1]\ldots [X_N, \sigma_N] = \K_{\Lambda}[X_1, \ldots, X_N] =: \oh_{q^{\Lambda}}(\K^N) \]
where $\Lambda:=(\lambda_{i,j})_{i,j} \in M_N(\K^*)$  with $\lambda_{j,i}=\lambda_{i,j}^{-1} $ for all $1\leq j\leq  i\leq N$. 

The following terms are needed to define our class of algebras of interest.

\begin{definition} Let $R$ be a $\K$-algebra and $R[x, \sigma, \delta]$ be an Ore extension.
	\begin{enumerate}
		\item We call the pair $(\sigma, \delta)$, and indeed the Ore extension, \emph{$q$-skew} if the automorphism and skew-derivation satisfy the relation $\delta \circ \sigma=q \sigma \circ \delta$, for some $1\neq q\in \K^*$.
		\item The skew-derivation $\delta$ is \emph{locally nilpotent} if, for every $0\neq a\in R$, there exists an integer $n_a\geq 0$ such that $\delta^{n_a}(a)=0$ and $\delta^m(a)\neq 0$ for any $m<n_a$.  We define such an $n_a$ as the \emph{$\delta$-nilpotence index of $a$}.
		\item \cite{Haynal} We say that the $\sigma$-derivation $\delta$ \emph{extends to a higher $q$-skew $\sigma$-derivation (h.$q$-s.$\sigma$-d.)} on $R$ if there is a sequence $\{d_n\}_{n=0}^{\infty}$ (denoted simply by $\{d_n\}$ when it is obvious which subscript indexes the sequence) of $\K$-linear operators on $R$ such that:
		\begin{itemize}
			\item $d_0$ is the identity;
			\item $d_1=\delta$;
			\item $d_n(rs)=\sum_{i=0}^n \sigma^{n-i}d_i(r)d_{n-i}(s)$ for all $r, s \in R$ and all $n$,
			\item $d_n\circ \sigma = q^n\sigma\circ d_n$ for all $n$.
		\end{itemize}
		A (h.$q$-s.$\sigma$-d.) is \emph{locally nilpotent} if, for all $0\neq r\in R$, there exists an integer $n_r\geq 0$ such that $d_n(r)=0$ for all $n\geq n_r$, and $d_m(r)\neq 0$ for any $m<n_r$. In this case we call $n_r$ the \emph{$d$-nilpotence index of $r$}. A h.$q$-s.$\sigma$-d is called \emph{iterative} if $d_n d_m=\binom{n+m}{m}_q d_{n+m}$ for all $n, m$. Here $\binom{n+m}{m}_q$ is the $q$-Gaussian binomial coefficient: a polynomial in $q$ over $\Z$ with properties similar to those of regular binomial coefficients. More precisely, for all $n \geq i \geq 0$, we have $\binom{n}{i}_q:=\frac{(n)!_q}{(i)!_q(n-i)!_q}$, where $(i)_q = \frac{q^i-1}{q-1}$ and $(i)!_q=(1)_q \cdots (i-1)_q(i)_q$ with the convention that $(0)!_q=1$. 
	\end{enumerate}
\end{definition}

The higher skew-derivations are what allowed Haynal to adapt Cauchon's \emph{effacement des derivations} homomorphism to a broader homomorphism which allows for roots of unity. For practical use, \cite[Theorem 2.8]{Haynal} gives a sufficient condition on a $q$-skew Ore extension $R[x; \sigma, \delta]$ for the $\sigma$-derivation $\delta$ to extend to a h.$q$-s.$\sigma$-d.\ for any $1\neq q \in \K^*$. This is used to check that the specific examples we are interested in do, indeed, satisfy this condition.

The algebras that we study in this article satisfy the following:
\begin{hypothesis}\label{hyp1} Let $A=\K[X_1][X_2;\sigma_2,\delta_2]\ldots [X_N;\sigma_N,\delta_N]$, where the $\sigma_i$ are $\K$-algebra automorphisms, and the $\delta_i$ are $\sigma_i$-derivations on the relevant subalgebras of $A$. Denote these subalgebras by $A_{i-1}:=\K[X_1][X_2;\sigma_2,\delta_2]\ldots [X_{i-1}; \sigma_{i-1}, \delta_{i-1}]$ so that $A_0:=\K$ and $A_N:=A$. We place the following conditions on $A$: 
	\begin{enumerate}[H\ref{hyp1}.1:]
		\item $\sigma_i(X_l)=\lambda_{i,l}X_l$ for all $l<i$ and $2\leq i \leq N$, where $\lambda_{i,l}\in \K^*$. \label{hyp1.2}
		\item $\Lambda:=(\lambda_{i,j})\in M_N(\K^*)$ is a multiplicatively antisymmetric matrix. That is, $\lambda_{i,l}=\lambda_{l,i}^{-1}$ for all $1\leq i, l \leq N$. \label{hyp1.3}
		\item For $2\leq i \leq N$ there exists some $1\neq q_i\in \K^*$ such that $\delta_i\circ \sigma_i=q_i\sigma_i\circ \delta_i$, i.e. $(\sigma_i,\delta_i)$ is $q_i$-skew.  \label{hyp1.4}							
		\item For all $2\leq i \leq N$ each $\delta_i$ extends to a locally nilpotent, iterative h.$q_i$-s.$\sigma_i$-d., $\{d_{i,n}\}_{n=0}^{\infty}$, on $A_{i-1}$, and $\sigma_l \circ d_{i,n}=\lambda_{l,i}^n d_{i,n} \circ \sigma_l$ on $A_{i-1}$ for all $n\geq 0 $ and $i+1\leq l \leq N$. \label{hyp1.7}
	\end{enumerate}
\end{hypothesis}

Algebras satisfying Hypothesis~\ref{hyp1} are called {\em quantum nilpotent algebras} and include quantum Schubert cell algebras, quantised Weyl algebras, quantised coordinate rings of affine, symplectic, and euclidean spaces, and quantum matrices. 

\section{Deleting Derivations Algorithm}\label{sectionDDA}
The main algorithm in this section will remove each skew-derivation from an algebra $A$ satisfying Hypothesis~\ref{hyp1} via an iterative change of indeterminates within the division ring of fractions of $A$, until we obtain a quantum affine space. This is developed in much the same way as in \cite[Section 3.2]{Cauchon} but applies the homomorphism defined in \cite[Proposition 3.4]{Haynal}. For this application we need the ability to reorder the extensions of $A$. This is achieved in the following lemma.
\begin{lemma}\label{mylemma}
Let
\begin{align*}
A&{}=A_{j-1}[X_j;\sigma_j,\delta_j][X_{j+1};\sigma_{j+1}]\ldots [X_N;\sigma_N],\\
\hat{A}&{}=A_{j-1}[X_j;\sigma_j,\delta_j][X_{j+1}^{\pm 1};\sigma_{j+1}]\ldots [X_N^{\pm 1};\sigma_N],
\end{align*}
where the automorphisms and skew-derivations satisfy properties H\ref{hyp1}.\ref{hyp1.2}--H\ref{hyp1}.\ref{hyp1.4} and $\delta_j\neq 0$.
\begin{enumerate}[(I)]
\item Then
\begin{align*}
A&{}=A_{j-1}[X_{j+1};\sigma^*_{j+1}]\ldots [X_N;\sigma^*_N][X_j;\sigma_j',\delta_j'],\\
\hat{A}&{}=A_{j-1}[X_{j+1}^{\pm 1};\sigma^*_{j+1}]\ldots [X_N^{\pm 1};\sigma^*_N][X_j;\sigma_j',\delta_j'],
\end{align*}
where
	\begin{enumerate}[(i)]
	\item $\sigma^*_i|_{A_{j-1}}=\sigma_i|_{A_{j-1}}$ for all $j+1\leq i\leq N$ and $\sigma^*_i(X_l)=\sigma_i(X_l)=\lambda_{i,l}X_l$ for all $j+1\leq l < i$;
	\item $\sigma_j'|_{A_{j-1}} = \sigma_j$ and $ \delta_j'|_{A_{j-1}} = \delta_j$;
	\item $\sigma_j'(X_l)=\lambda_{j,l}X_l=\lambda_{l,j}^{-1}X_l$ and $\delta_j'(X_l)=0$ for all $j+1 \leq l \leq N$.
	\end{enumerate}
\item $(\sigma_j',\delta_j')$ is $q_j$-skew.
\item Suppose that property H\ref{hyp1}.\ref{hyp1.7} is also satisfied. Then $\delta_j'$ extends to a locally nilpotent, iterative h.$q_j$-s.$\sigma_j'$-d., $\{d'_{j,n}\}_{n=0}^{\infty}$, on $A_{j-1}\langle X_{j+1}^{\pm 1}, X_{j+2}^{\pm 1}, \ldots, X_N^{\pm 1} \rangle$, where the $d_{j,n}'$ coincide with the $d_{j,n}$ on $A_{j-1}$ and,  for all $j+1\leq l \leq N$ and $n\geq 1$,  $d_{j,n}'(X_l)=0$. Moreover, $\{d_{j,n}'\}_{n=0}^{\infty}$ restricts to a h.$q_j$-s.$\sigma_j'$-d. on $A_{j-1}\langle X_{j+1}, X_{j+2}, \ldots, X_N \rangle$ which is also locally nilpotent and iterative.
\end{enumerate}
\end{lemma}
\begin{proof} This is an inductive corollary to \cite[Lemma 4.1]{Haynal} and the details are left to the reader.
\end{proof}

In the results that follow, we abuse notation slightly and use the same notation for maps defined on isomorphic algebras. We do this in the case where the action of the map on the generators of the algebra does not change, even though the algebras do. For example, suppose we have two isomorphic iterated Ore extensions
\begin{align*}
\K[X_1][X_2;\sigma_2, \delta_2] \cong \K[x_1][x_2;\bar{\sigma}_2, \bar{\delta}_2],
\end{align*}
where $\sigma_2(X_1)=\lambda X_1$ and $\bar{\sigma}_2(x_1)=\lambda x_1$ for the same $\lambda\in \K^*$. In this case, we simply denote $\bar{\sigma}_2$ by $\sigma_2$ (similarly for $\bar{\delta}_2$) and write $\K[x_1][x_2; \sigma_2, \delta_2]$.  

We may also abuse notation in a similar way for restrictions of maps to isomorphic subalgebras. 

We now describe the deleting derivations algorithm for algebras satisfying Hypothesis~\ref{hyp1}.

\subsection{The Algorithm}

For each $j\in \llbracket 2, N+1 \rrbracket$ we define a sequence $(X_1^{(j)},\ldots,X_N^{(j)})$ of elements of $F:=\Frac(A)$ and we set $A^{(j)}:=\K\langle X_1^{(j)},\ldots,X_N^{(j)} \rangle$ to be the subalgebra of $F$ generated by these elements.  For $j=N+1$ we set $(X_1^{(N+1)},\ldots, X_N^{(N+1)}):=(X_1,\ldots , X_N)$ so that $A^{(N+1)}=A$. For some fixed  $j \in \llbracket 2, N \rrbracket$ we assume that the algebra $A^{(j+1)}$ satisfies the following hypothesis, first setting $(x_1,\ldots , x_N):=(X_1^{(j+1)},\ldots,X_N^{(j+1)})$ for ease of notation.
\begin{hypothesis}\label{hyp2}
\ \begin{enumerate}[H\ref{hyp2}.1:]
\item $A^{(j+1)}\cong \K[X_1]\ldots[X_j;\sigma_j,\delta_j][X_{j+1};\sigma_{j+1}^{(j+1)}]\ldots[X_N;\sigma_N^{(j+1)}]$ by an isomorphism sending $x_i$ to $ X_i$ for all $i\in \llbracket 1, N\rrbracket$.  \label{hyp2.1}
\item For each $i \in \llbracket j+1, N \rrbracket$, the map $\sigma_i^{(j+1)}$ is an automorphism such that $\sigma_i^{(j+1)}(X_l)=\lambda_{i,l}X_l$ for all $l \in \llbracket 1, i-1 \rrbracket$. Furthermore, we have $\sigma_i^{(j+1)} \circ d_{l,n}=\lambda_{i,l}^n d_{l,n}\circ \sigma_i^{(j+1)}$ for all $l \in \llbracket 2, j \rrbracket$ and $n\geq 0$.\label{hyp2.2}
\end{enumerate}
\end{hypothesis}
This allows us to write
\begin{align}
A^{(j+1)}&{}=\K\langle x_1,\ldots,x_N \rangle \\ &{}= \K[x_1]\ldots[x_j;\sigma_j,\delta_j][x_{j+1};\sigma_{j+1}^{(j+1)}]\ldots[x_N;\sigma_N^{(j+1)}] \label{Xmap}
\end{align}
where, for $i \in \llbracket 2, j \rrbracket$, $\sigma_i$ and $\delta_i$ satisfy Hypothesis \ref{hyp1} and $\delta_i$ extends to a locally nilpotent, iterative h.$q_i$-s.$\sigma_i$-d., $\{d_{i,n}\}_{n=0}^{\infty}$, on $A_{i-1}$.

We define a new sequence of elements in $F$, $(y_1, \ldots, y_N):=(X_1^{(j)}, \ldots, X_N^{(j)})$:
\begin{equation}\label{yequation}
y_l=
\begin{cases}
x_l & l\geq j; \\
\sum_{n=0}^{\infty} q_j^{\frac{n(n+1)}{2}}(q_j-1)^{-n}d_{j,n}\circ \sigma_j^{-n}(x_l)x_j^{-n} & l<j.
\end{cases}
\end{equation}
Note that the sum stated above is finite, since the sequence $\{d_{j,n}\}_{n=0}^{\infty}$ is locally nilpotent. With this we define $A^{(j)}:=\K\langle y_1,\ldots,y_N\rangle$.

\begin{theorem}\label{mainthm}
Let $A$ be as in Hypothesis \ref{hyp1}, with $A^{(j+1)}$ defined as above and satisfying Hypothesis \ref{hyp2}. Then we have the following:
\begin{enumerate}[(I)]
\item \label{mainthmIsom} $A^{(j)} \cong \K[X_1][X_2;\sigma_2,\delta_2]\cdots [X_{j-1};\sigma_{j-1},\delta_{j-1}][X_j;\sigma_j^{(j)}]\cdots [X_N;\sigma_N^{(j)}]$ by an isomorphism which sends $y_l$ to $X_l$ for all $l\in \llbracket 1, N \rrbracket$.
\item For all $i\in \llbracket j, N \rrbracket$, the $\sigma_i^{(j)}$ are automorphisms satisfying
	\begin{enumerate}[(i)]
		\item $\sigma_i^{(j)}(X_l)=\lambda_{i,l}X_l$ for all $l \in \llbracket 1, i-1 \rrbracket$;
		\item $\sigma_i^{(j)} \circ d_{l,n}=\lambda_{i,l}^n d_{l,n} \circ \sigma_i^{(j)}$  for all $n\geq 0$ and all $l \in \llbracket 2, j-1 \rrbracket$.
	\end{enumerate}
\item Let $S_j:=\{x_j^n\mid n\geq 0\}=\{y_j^n\mid n\geq 0\}$. This is a multiplicatively closed set of regular elements satisfying the two sided Ore-condition in $A^{(j+1)}$ and $A^{(j)}$ and, furthermore, $A^{(j)}S_j^{-1}=A^{(j+1)}S_j^{-1}$. \label{mainthmSj}
\end{enumerate}
\end{theorem}
\begin{proof}
By Hypothesis \ref{hyp2}, and the discussion thereafter, we can write $A^{(j+1)}$ as
\[A^{(j+1)}=\K[x_1]\ldots[x_j;\sigma_j,\delta_j][x_{j+1};\sigma_{j+1}^{(j+1)}]\ldots[x_N;\sigma_N^{(j+1)}].\]
Define
\[A^{(j+1)}_{j-1}:=\K[x_1][x_2;\sigma_2,\delta_2]\ldots [x_{j-1};\sigma_{j-1},\delta_{j-1}].\]
If $\delta_j=0$ then $A^{(j)}=A^{(j+1)}$ and the result is trivial. So, assuming that $\delta_j\neq0$ and applying Lemma~\ref{mylemma} to $A^{(j+1)}=A^{(j+1)}_{j-1}[x_j;\sigma_j,\delta_j][x_{j+1};\sigma_{j+1}^{(j+1)}]\ldots[x_N;\sigma_N^{(j+1)}]$ gives
\begin{equation}\label{eqnCj+1}
A^{(j+1)}=A^{(j+1)}_{j-1}[x_{j+1};\sigma_{j+1}^{(j+1)*}]\ldots[x_N;\sigma_N^{(j+1)*}][x_j;\sigma'_j,\delta'_j],
\end{equation}
where
\begin{enumerate}[(a)]
\item $\sigma_i^{(j+1)*}|_{A^{(j+1)}_{j-1}}=\sigma_i^{(j+1)}|_{A^{(j+1)}_{j-1}}$ for all $i \in \llbracket j+1, N \rrbracket$, and $\sigma_i^{(j+1)*}(x_l)=\sigma_i^{(j+1)}(x_l)=\lambda_{i,l}x_l$ for all $l \in \llbracket j+1, i-1 \rrbracket $;
\item $\sigma'_j|_{A^{(j+1)}_{j-1}}=\sigma_j$ and $\delta'_j|_{A^{(j+1)}_{j-1}}=\delta_j$;
\item $\sigma'_j(x_l)=\lambda_{j,l}x_l=\lambda_{l,j}^{-1}x_l$ and $\delta'_j(x_l)=0$ for all $l \in \llbracket j+1, N \rrbracket$.
\end{enumerate}
In particular, $\delta'_j$ extends to a h.$q_j$-s.$\sigma'_j$-d., $\{d'_{j,n}\}_{n=0}^{\infty}$, on $A^{(j+1)}_{j-1}\langle x_{j+1}^{\pm 1},\ldots, x_N^{\pm 1} \rangle$ where
\begin{align}
d'_{j,n}|_{A^{(j+1)}_{j-1}}&{}=d_{j,n} \quad (\forall \; n\geq 0),  \label{skewderivrestrict1}\\
d'_{j,n}(x_l)&{}=0 \qquad (\forall \; l \in \llbracket j+1, N \rrbracket \text{ and } n\geq 1). \label{skewderivrestrict2}
\end{align}
Define
\[\widehat{A^{(j+1)}}:=A^{(j+1)}_{j-1}[x_{j+1};\sigma_{j+1}^{(j+1)*}]\ldots[x_N;\sigma_N^{(j+1)*}]\]
so that equation (\ref{eqnCj+1}) becomes
\[A^{(j+1)}=\widehat{A^{(j+1)}}[x_j;\sigma'_j,\delta'_j]. \]
Applying \cite[Theorem 3.7]{Haynal} to $A^{(j+1)}$ yields the isomorphism,
\begin{align*}
f:\widehat{A^{(j+1)}}[x_j^{\pm 1};\sigma'_j] &{} ~ \longrightarrow ~ \widehat{A^{(j+1)}}[x_j;\sigma'_j,\delta'_j]S_j^{-1}\\
\widehat{A^{(j+1)}} ~ \ni ~ a &{} ~ \longmapsto ~ f(a)=\sum_{n=0}^{\infty} q_j^{\frac{n(n+1)}{2}}(q_j-1)^{-n}d'_{j,n} \circ (\sigma'_j)^{-n}(a)x_j^{-n} \\
x_j &{} ~ \longmapsto ~ x_j.
\end{align*}
Note that $f(x_l)=x_l$ for  all $l \in \llbracket j+1, N \rrbracket$ since, by (\ref{skewderivrestrict2}), $d'_{j,n}(x_l)=0$ for all $n\geq 1$.
If $l \in \llbracket 1, j-1 \rrbracket$, then $x_l$ and $(\sigma'_j)^{-n}(x_l)\in A^{(j+1)}_{j-1}$. Thus, by (\ref{skewderivrestrict1}) and (b) above, we can replace $\sigma'_j$ and $d'_{j,n}$ in $f(x_l)$ with $\sigma_j$ and $d_{j,n}$ to obtain
\[f(x_l)=\sum_{n=0}^{\infty} q_j^{\frac{n(n+1)}{2}}(q_j-1)^{-n}d_{j,n}\circ \sigma_j^{-n}(x_l)x_j^{-n}=y_l,\]
as defined in (\ref{yequation}).  Therefore, for any $l \in \llbracket 1, N \rrbracket$ we see that
\[f(x_l)=\begin{cases}
            x_l & l\geq j; \\
            y_l & l<j.
        \end{cases}\]
Hence the isomorphism $f$ takes $x_l$ to $y_l$ for all $l \in \llbracket 1, N \rrbracket$.

Using \cite[Theorem 3.7]{Haynal} we see that restricting $f$ to $\widehat{A^{(j+1)}}[x_j;\sigma'_j]$ gives the deleting derivation homomorphism as defined in \cite[Proposition 3.4]{Haynal} . Therefore
\begin{equation}\label{imageisom}
\Ima(f)\cong \widehat{A^{(j+1)}}[x_j;\sigma'_j],
\end{equation}
where $\Ima(f)=f(\widehat{A^{(j+1)}}[x_j;\sigma'_j])$ is the subalgebra of $\widehat{A^{(j+1)}}[x_j^{\pm 1};\sigma'_j, \delta'_j]$ generated by $x_j$ and $f(\widehat{A^{(j+1)}})$.  Since $f(\widehat{A^{(j+1)}})$ is generated by $\K$ and $\{y_l\mid l\neq j\}$, and since $x_j=y_j$, this simply tells us that
\[\Ima(f)=\K\langle y_1, \ldots , y_N\rangle = A^{(j)}.\]
Using (\ref{imageisom}) we see that
\[A^{(j)}=\Ima(f)\cong \K[x_1]\ldots[x_{j-1};\sigma_{j-1},\delta_{j-1}][x_{j+1};\sigma_{j+1}^{(j+1)*}]\ldots[x_N;\sigma_N^{(j+1)*}][x_j;\sigma'_j],\]
and therefore,
\begin{equation}\label{cjeqn}
A^{(j)}=\K[y_1]\ldots [y_{j-1};\sigma_{j-1},\delta_{j-1}][y_{j+1};\sigma_{j+1}^{(j+1)*}]\ldots [y_N;\sigma_N^{(j+1)*}][y_j;\sigma'_j].
\end{equation}
Finally we apply \cite[Proposition 3.6]{Haynal} to conclude that $S_j$ is a multiplicatively closed set of regular elements in both $A^{(j+1)}$ and $A^{(j)}$, satisfying the two-sided Ore condition, and that
\begin{align*}
\Ima(f)S_j^{-1} &{}= \widehat{A^{(j+1)}}[x_j;\sigma'_j,\delta'_j]S_j^{-1},\\
A^{(j)} S_j^{-1} &{}= A^{(j+1)}S_j^{-1}.
\end{align*}
Thus assertion (\ref{mainthmSj}) is proved.

The property $y_ly_j=\lambda_{l,j}y_jy_l$, along with the fact that $\lambda_{l,j}=\lambda_{j,l}^{-1}$, allows us to rearrange (\ref{cjeqn}) to obtain
\begin{equation}\label{cjrearranged}
A^{(j)}=\K[y_1]\ldots [y_{j-1};\sigma_{j-1},\delta_{j-1}][y_j;\sigma^{(j)}_j][y_{j+1};\sigma^{(j)}_{j+1}]\ldots [y_N;\sigma^{(j)}_N].
\end{equation}
Defining
\[A^{(j)}_{j-1}:=\K[y_1]\ldots [y_{j-1};\sigma_{j-1},\delta_{j-1}] \]
we see that $A_{j-1}^{(j)} \cong A_{j-1}$ and there is an isomorphism\begin{align}
A^{(j)}&{}=\K[y_1]\ldots [y_{j-1};\sigma_{j-1},\delta_{j-1}][y_j;\sigma^{(j)}_j][y_{j+1};\sigma^{(j)}_{j+1}]\ldots [y_N;\sigma^{(j)}_N] \nonumber\\
&{}\cong  \K[X_1]\ldots [X_{j-1};\sigma_{j-1},\delta_{j-1}][X_j;\sigma^{(j)}_j][X_{j+1};\sigma^{(j)}_{j+1}]\ldots [X_N;\sigma^{(j)}_N] \label{assertion1}
\end{align}
sending $y_l$ to $X_l$ for all $l \in \llbracket 1, N \rrbracket$, where the maps (as defined on suitable subalgebras of $A^{(j)}$) are as follows:
\begin{enumerate}[(a$'$)]
    \item $\sigma^{(j)}_j=\sigma'_j|_{A^{(j)}_{j-1}}=\sigma_j$;
    \item $\sigma_i^{(j)}|_{A^{(j)}_{j-1}}=\sigma_i^{(j+1)*}|_{A^{(j)}_{j-1}}=\sigma_i^{(j+1)}|_{A^{(j)}_{j-1}}$ for all $i \in \llbracket j+1, N \rrbracket$;
    \item $\sigma^{(j)}_i(y_l)=\lambda_{i,l}y_l$ for all $i \in \llbracket j+1, N \rrbracket$ and $l \in \llbracket 1, i-1 \rrbracket$.
\end{enumerate}
Using the isomorphism in (\ref{assertion1}) along with the observations (a$'$)--(c$'$) above we can prove assertion (II) for all $i \in \llbracket j, N \rrbracket$: Observation (a$'$) proves both parts of assertion (II) when $i=j$, since $\sigma_j$ satisfies assertion (II) by definition (see H\ref{hyp1}.\ref{hyp1.2} and H\ref{hyp1}.\ref{hyp1.7}). When $i\in \llbracket j+1, N \rrbracket$, observation (b$'$) proves (II)(ii), since $\sigma_i^{(j+1)}$ satisfies H\ref{hyp2}.\ref{hyp2.2}, and observation (c$'$) proves (II)(i).

\end{proof}

\begin{remark}
If $A$ is an algebra satisfying Hypothesis \ref{hyp1} then Hypothesis \ref{hyp2} is satisfied for $j=N+1$. Theorem \ref{mainthm} then implies that Hypothesis \ref{hyp2} is also satisfied for all $j \in \llbracket 2, N+1 \rrbracket$. 
\end{remark}

\begin{corollary}\label{maincorollary}
The algebra $A':=A^{(2)}$ is a quantum affine space. More precisely, by setting $T_i:=X_i^{(2)}$ for all $i\in \llbracket 1, N \rrbracket$ and  $\Lambda:=(\lambda_{i,j})\in M_N(\K)$ to be the multiplicatively antisymmetric matrix, we obtain:
\[ A'=\K_{\Lambda}[T_1,\ldots,T_N]. \]
\end{corollary}

\begin{remark}
For all $j\in \llbracket 1, N\rrbracket$, we say that $A^{(j+1)}$ is the algebra obtained from $A$ after $N-j$ steps of the deleting derivations algorithm.
\end{remark}

\subsection{Ring of fractions}
In order to be able to track the completely prime ideals along the deleting derivations algorithm we need the following two results regarding the division ring of fractions of the algebras $A^{(j)}$ at each step of the algorithm.  These results were proved in \cite[Subsection 3.3]{Cauchon} in the generic setting and can be applied directly to our setting thanks to the above results.

Let $\Sigma$ be the multiplicatively closed set in $A'$ generated by the elements $T_1,\ldots, T_N$. For $j \in \llbracket 2, N \rrbracket$, define sets $\Sigma_j$:
\begin{align*}
\Sigma_2&{}:=\Sigma,\\
\Sigma_{j+1}&{}=A^{(j+1)}\cap \Sigma_j \quad \text{ for } j\in \llbracket 2, N \rrbracket.
\end{align*}
\begin{proposition}\label{propconfusedalgs}
For all $j\in \llbracket 2, N+1 \rrbracket$ the following are true:
\begin{enumerate}[(i)]
\item \label{confusedalg1}  $\Sigma_j$ is a multiplicatively closed set of regular elements in $A^{(j)}$ containing $X_{j-1}^{(j)},\ldots, X_N^{(j)}$;
\item \label{confusedalg2} $\Sigma_j$ satisfies the two-sided Ore condition in $A^{(j)}$;
\item \label{confusedalg3} The algebras $A^{(j)}\Sigma_j^{-1} \subset \Frac(A)$ are all equal.
\end{enumerate}
\end{proposition}
\begin{proof}
We proceed by induction on $j$.  When $j=2$, statements (\ref{confusedalg1}) and (\ref{confusedalg2}) are trivially true from the definition of $\Sigma$ and the fact that the generators (and monomials in these generators) of a quantum affine space are regular and normal.
Let $j\in \llbracket 2, N\rrbracket$ and suppose statements (\ref{confusedalg1}) and (\ref{confusedalg2}) hold for $j$. We will show they also hold for $j+1$ and that $A^{(j)}\Sigma_j^{-1}=A^{(j+1)}\Sigma_{j+1}^{-1}$.

Recall the notation from the previous section: $\{x_i\}_i = \{X_i^{(j+1)}\}_i$ and $\{y_i\}_i = \{X_i^{(j)}\}_i$ where $x_i=y_i$ for all $i\geq j$. By the induction hypothesis, $\Sigma_j$ is a multiplicatively closed set of regular elements in $A^{(j)}$ containing $y_{j-1},\ldots, y_N$. Therefore $\Sigma_{j+1}=A^{(j)} \cap \Sigma_j$ is a multiplicatively closed set of regular elements of $A^{(j+1)}$ containing $y_j=x_j, \; \ldots, \; y_N=x_N$, thus proving statement (\ref{confusedalg1}).

Recall the set $S_j=\{x_j^n \mid n\in \N\}=\{y_j^n \mid n\in \N\}\subset \Sigma_j \cap \Sigma_{j+1}$ and use Theorem \ref{mainthm}(\ref{mainthmSj}) to obtain the inclusions
\begin{align}\label{eqnconfusedalg}
A^{(j+1)} \subset A^{(j+1)}S_j^{-1}=A^{(j)}S_j^{-1}\subset A^{(j)}\Sigma_j^{-1}.
\end{align}
Since $\Sigma_{j+1}\subset \Sigma_j$ then  $\Sigma_{j+1}$ must be invertible in $A^{(j)}\Sigma_j^{-1}$. We use this to show that an element $a\in A^{(j)}\Sigma_j^{-1}$ can be rewritten as an element in $A^{(j+1)}\Sigma_{j+1}^{-1}$. Write $a=yu^{-1}$, with $y\in A^{(j)}$ and $u\in \Sigma_j$.  Since
\[\Sigma_j \subset A^{(j)} \subset A^{(j)}S_j^{-1} = A^{(j+1)}S_j^{-1}, \]
we can write $u=vs_1^{-1}$ and $y=xs_2^{-1}$, with $v, x \in A^{(j+1)}$ and $s_1, s_2 \in S_j$. Then $a$ becomes
\[a=xs_2^{-1}(vs_1^{-1})^{-1}=xs_2^{-1}s_1v^{-1}=xs_1s_2^{-1}v^{-1}=xs_1(vs_2)^{-1}.\]
Observe that $vs_2=(vs_1^{-1})s_1s_2=us_1s_2 \in \Sigma_j$ since $u\in \Sigma_j$ and $s_1,s_2\in S_j \subset \Sigma_j$.  Also, $v\in A^{(j+1)}$ and $s_2\in S_j \subset A^{(j+1)}$, therefore $vs_2\in A^{(j+1)}\cap \Sigma_j = \Sigma_{j+1}$.  Similarly, $xs_1\in A^{(j+1)}$ so we can write
\[a=bc^{-1} \in A^{(j+1)}\Sigma_{j+1}^{-1},\]
where $b=xs_1\in A^{(j+1)}$ and $c=vs_2\in \Sigma_{j+1}$. From the inductive hypothesis we know that $A^{(j)}\Sigma_j^{-1}=\Sigma_j^{-1}A^{(j)}$, so if $a\in A^{(j)}\Sigma_j^{-1}$ then it must also be true that $a\in \Sigma_j^{-1}A^{(j)}$. We also know from Theorem \ref{mainthm}(\ref{mainthmSj}) that $S_j$ is an Ore set in $A^{(j)}$ and $A^{(j+1)}$, thus $A^{(j)}S_j^{-1}=S_j^{-1}A^{(j)}$ and $A^{(j+1)}S_j^{-1}=S_j^{-1}A^{(j+1)}$. Using these results we can follow a similar method to before to rewrite $a\in A^{(j)}\Sigma_j^{-1}=\Sigma_j^{-1}A^{(j)}$ as
\[ a=c'^{-1}b' \in \Sigma_{j+1}^{-1}A^{(j+1)}, \]
with $c'\in \Sigma_{j+1}$ and $b'\in A^{(j+1)}$.

If we can prove that $\Sigma_{j+1}$ is a two-sided Ore set in $A^{(j+1)}$ then the working above implies that $A^{(j)}\Sigma_j^{-1}\subseteq A^{(j+1)}\Sigma_{j+1}^{-1}$. Furthermore, $A^{(j+1)}\subset A^{(j)}\Sigma_j^{-1}$ and $\Sigma_{j+1}\subset \Sigma_j$ so we also have $A^{(j+1)}\Sigma_{j+1}^{-1}\subseteq A^{(j)}\Sigma_j^{-1}$. Hence statement (\ref{confusedalg3}) is true if we can prove that statement (\ref{confusedalg2}) holds.

From the inclusion $A^{(j+1)}\Sigma_{j+1}^{-1}\subseteq A^{(j)}\Sigma_j^{-1}$ we can write any $a=bc^{-1}\in A^{(j+1)}\Sigma_{j+1}^{-1}$ as $a\in A^{(j)} \Sigma_j^{-1}$ and, applying the above working, we see that there exist $c'\in \Sigma_{j+1}$ and $b'\in A^{(j+1)}$ such that $a=c'^{-1}b'\in \Sigma_{j+1}^{-1}A^{(j+1)}$. This verifies the two-sided Ore condition on $\Sigma_{j+1}\subset A^{(j+1)}$ necessary for proving statement (\ref{confusedalg2}) and, by the comment earlier, statement (\ref{confusedalg3}). 
\end{proof}

From the above proposition it is clear that:
\begin{theorem} \label{thmPIdegs}
\begin{enumerate}[(i)]
\item There exists a multiplicatively closed set of regular elements $S\subseteq A$ such that $AS^{-1}=A'\Sigma^{-1}=\K_{\Lambda}[T_1^{\pm 1}, \ldots, T_N^{\pm 1}]$. 
\item $\Frac(A^{(j)})=\Frac(A)$ for all $j\in \llbracket 2, N+1\rrbracket$ and, in particular, $\Frac(A)=\Frac(A')$.
\item $A$ is a PI algebra if and only if the $\lambda_{i,j}$ are roots of unity for all $i, j \in \llbracket 1, N \rrbracket$ and, in this case, $\PI(A)=\PI(A')$.
\end{enumerate}
\end{theorem}
\begin{proof}
Let $S=\Sigma_{N+1}$. Proposition \ref{propconfusedalgs}(\ref{confusedalg3}) shows  that $AS^{-1}=A'\Sigma^{-1}$ and that all the algebras $A^{(j)}$ have a common localisation, thus proving (i) and (ii). \cite[Proposition 7.1]{DeConciniProcesi} states that $A'$ is PI if and only if the $\lambda_{i,j}$ are roots of unity for all $1\leq i, j \leq N$ and \cite[Corollary 4.7]{Haynal} states the same result for $A$. Therefore they have the same PI degree since they have equal total rings of fractions. This proves (iii).
\end{proof}

\section{Deleting Derivations Algorithm on completely prime quotients}
In this section we set up a canonical embedding $\psi:\CSpec(A) \rightarrow \CSpec(A')$, from the completely prime spectrum of $A$ into the completely prime spectrum of $A'$, and use this to extend the algorithm of the previous section to quotient algebras. We also define criteria for some ideal $Q\in \CSpec(A')$ to lie in the image of $\psi$.

Many of the results of this section are analogues to those found in the generic setting \cite[Sections 4 and 5]{Cauchon} and their proofs follow in almost the same way, thanks to the results of the previous section. The proofs concerning membership criteria for the canonical embedding in Subsection~\ref{sectionProperties} closely mirror those in the Poisson setting, which can be found in \cite[Sections 2.3 and 2.4]{LaunoisLecoutre}.

\subsection{Canonical Embedding and partition of completely prime spectra}
\subsubsection{Embedding}
Let $A$ satisfy Hypothesis~\ref{hyp1} and define the following sets, which we endow with the induced Zariski topology:
\begin{align*}
\mathcal{P}_j^0(A^{(j)}) &{}:=\{P\in \CSpec(A^{(j)}) \mid y_j\notin P\},\\
\mathcal{P}_j^1(A^{(j)})&{}:=\{P\in \CSpec(A^{(j)}) \mid y_j\in P\}, \\
\mathcal{P}_j^0(A^{(j+1)})&{}:=\{P\in \CSpec(A^{(j+1)})\mid x_j\notin P\},\\
\mathcal{P}_j^1(A^{(j+1)})&{}:=\{P\in \CSpec(A^{(j+1)})\mid  x_j\in P\}.
\end{align*}

In the following results, by the term \emph{increasing} we mean that the homeomorphism is order-preserving with respect to inclusion of ideals. By \emph{bi-increasing} we mean that both the homeomorphism and its inverse are increasing homeomorphisms.

\begin{lemma} \label{lemBiIncHomP0}
There is an increasing homeomorphism $\psi^0_j:\mathcal{P}^0_j(A^{(j+1)}) \rightarrow \mathcal{P}^0_j(A^{(j)})$ given by $\psi^0_j(P):=PS_j^{-1}\cap A^{(j)}$.  Its inverse is defined by $(\psi^0_j)^{-1}(Q):=QS_j^{-1}\cap A^{(j+1)}$ and is also increasing. Hence $\psi^0_j$ is bi-increasing.
\end{lemma}
\begin{proof}
Extension and contraction maps provide bi-increasing inverse homeomorphisms between $\CSpec(A^{(j)}S_j^{-1})$ and $\mathcal{P}_j^0(A^{(j)})$. Similarly, we obtain a bi-increasing homeomorphism between $\mathcal{P}_j^{(0)}(A^{(j+1)})$ and $\CSpec(A^{(j+1)}S_j^{-1})$. Since $A^{(j)}S_j^{-1} = A^{(j+1)}S_j^{-1}$ then their completely prime spectra, as topological spaces, are equal. Therefore the two bi-increasing homeomorphisms defined here give rise to the bi-increasing homeomorphism in the statement of this lemma.
\end{proof}

Next we turn our attention to the sets $\mathcal{P}^1_j(A^{(j)})$ and $\mathcal{P}^1_j(A^{(j+1)})$.
\begin{lemma}\label{lemMapg_j}
There is a surjective algebra homomorphism $g_j:A^{(j)} \rightarrow A^{(j+1)}/\langle x_j \rangle$ which takes $y_i\mapsto x_i+\langle x_j \rangle$, for all $1\leq i \leq N$, where $x_i+\langle x_j \rangle$ is the canonical image of $x_i$ in $A^{(j+1)}/\langle x_j \rangle $.
\end{lemma}
\begin{proof}
By Theorem \ref{mainthm} we have
\begin{align*}
A^{(j+1)} &{}:= \K\langle x_1, \ldots, x_N\rangle \cong \K[X_1]\cdots [X_j;\sigma_j,\delta_j][X_{j+1};\sigma^{(j+1)}_{j+1}]\cdots [X_N;\sigma^{(j+1)}_N],\\
A^{(j)} &{}:= \K\langle y_1, \ldots, y_N\rangle \cong \K[X_1]\cdots [X_{j-1};\sigma_{j-1},\delta_{j-1}][X_j;\sigma^{(j)}_j]\cdots [X_N;\sigma^{(j)}_N].
\end{align*}
Restricting these algebras to $R:=\K\langle x_1, \ldots, x_{j-1}\rangle$ and $S:=\K\langle y_1, \ldots, y_{j-1}\rangle$ we see that there is an isomorphism $S\rightarrow R$ sending $y_i\mapsto x_i$ for all $i\in \llbracket 1, j-1 \rrbracket$. Since $R\subseteq A^{(j+1)}$ we can compose this isomorphism with the natural surjection $A^{(j+1)}\rightarrow A^{(j+1)}/\langle x_j \rangle$ to obtain the algebra homomorphism $f:S\rightarrow A^{(j+1)}/\langle x_j \rangle$, sending $y_i\mapsto x_i+\langle x_j \rangle$ for all  $i\in \llbracket 1, j-1 \rrbracket$.  Using the  the commutation rules for $A^{(j+1)}$ and the fact that $x_j+\langle x_j \rangle=0$, we see that $(x_j+\langle x_j \rangle)(x_i+\langle x_j \rangle) = \lambda_{j,i} (x_i+\langle x_j \rangle)(x_j+\langle x_j \rangle)=0$ for all $i\in \llbracket 1, j-1 \rrbracket$. The relations on $x_i+\langle x_j \rangle\in A^{(j+1)}/\langle x_j \rangle$ therefore agree with those on $y_i\in S[y_j;\sigma^{(j)}_j]\cdots [y_N;\sigma^{(j)}_N]=A^{(j)}$, for all $i\in \llbracket 1, N \rrbracket$. Applying the universal property of Ore extensions allows us to conclude the construction of the surjective homomorphism $g_j$.
\end{proof}

\begin{lemma}\label{lemPsi1}
There is an increasing injective map $\psi_j^1:\mathcal{P}^1_j(A^{(j+1)})\rightarrow \mathcal{P}^1_j(A^{(j)})$ taking $P\mapsto \psi_j^1(P):=g_j^{-1}(P/\langle x_j \rangle)$, where $P/\langle x_j \rangle$ denotes the canonical image of $P$ in $A^{(j+1)}/\langle x_j \rangle$, which induces a bi-increasing homeomorphism between $\mathcal{P}_j^1(A^{(j+1)})$ and the image $\psi_j^1(\mathcal{P}_j^1(A^{(j+1)}))$.
\end{lemma}
\begin{proof}
Using the First Isomorphism Theorem for algebras, we restrict the map $g_j$ from Lemma \ref{lemMapg_j} to yield an isomorphism $g'_j: A^{(j)}/\ker(g_j) \isomarrow A^{(j+1)}/\langle x_j \rangle$. This induces the following bi-increasing homeomorphisms between sets endowed with the Zariski topology:
\begin{align*}
f_1: \mathcal{P}_j^1(A^{(j+1)}) &{} ~ \longrightarrow ~ \mathrm{C.Spec}(A^{(j+1)}/\langle x_j \rangle),\\
f_2: \mathrm{C.Spec}(A^{(j+1)}/\langle x_j \rangle )  &{} ~ \longrightarrow ~ \mathrm{C.Spec}(A^{(j)}/\ker(g_j)), \\
f_3: \mathrm{C.Spec}(A^{(j)}/\ker(g_j))  &{} ~ \longrightarrow ~ \{Q\in \mathrm{C.Spec}(A^{(j)}) \mid \ker(g_j)\subseteq Q \}.
\end{align*}
The composition of these maps gives a bi-increasing homeomorphism
\begin{align*}
f_3 \circ f_2 \circ f_1 : \mathcal{P}_j^1(A^{(j+1)}) &{} ~ \longrightarrow ~ \{Q\in \mathrm{C.Spec}(A^{(j)}) \mid \ker(g_j)\subseteq Q \} \\
P &{} ~ \longmapsto ~ g_j^{-1}(P/\langle x_j \rangle).
\end{align*}
Note that $g_j(y_j)=x_j+\langle x_j \rangle=0$ so $\langle y_j \rangle \subseteq \ker(g_j)$, which leads to the inclusion
\[\{Q\in \mathrm{C.Spec}(A^{(j)}) \mid \ker(g_j) \subseteq Q \} ~ \subseteq ~ \mathcal{P}_j^1(A^{(j)}).\]
Therefore, from $f_3 \circ f_2 \circ f_1$, we can define an increasing injective map
\begin{align*}
\psi_j^1: \mathcal{P}_j^1(A^{(j+1)}) &{} ~ \longrightarrow ~ \mathcal{P}_j^1(A^{(j)}) \\
P &{} ~ \longmapsto ~ g_j^{-1}(P/\langle x_j \rangle),
\end{align*}
which induces a bi-increasing homeomorphism on its image, $\{Q\in \mathrm{C.Spec}(A^{(j)}) \mid \ker(g_j)\subseteq Q \}$.
\end{proof}

Using the two previous results we define the map $\psi_j:\CSpec(A^{(j+1)}) \rightarrow \CSpec(A^{(j)})$ where, for $P\in \CSpec(A^{(j+1)})$, we set
\[ \psi_j(P) :=\begin{cases}
	\psi_j^0(P)=PS_j^{-1}\cap A^{(j)} & \text{if } P\in \mathcal{P}_j^0(A^{(j+1)}); \\
	\psi_j^1(P)=g_j^{-1}(P/\langle x_j \rangle) & \text{if } P\in \mathcal{P}_j^1(A^{(j+1)}).
\end{cases} \]
The next result follows immediately.

\begin{proposition}\label{lemBiHom}
For $j\in \llbracket 2, N \rrbracket$ the map $\psi_j:\CSpec(A^{(j+1)}) \rightarrow \CSpec(A^{(j)})$ is injective. For $\epsilon \in \{0,1\}$, $\psi_j$ induces (by restriction) a bi-increasing homeomorphism $\mathcal{P}_j^{\epsilon}(A^{(j+1)}) \rightarrow \psi_j(\mathcal{P}_j^{\epsilon}(A^{(j+1)}))$ which is a closed subset of $\mathcal{P}_j^{\epsilon}(A^{(j)})$.
\end{proposition}

Using these maps, we define the canonical embedding:
\begin{definition}\label{defnPsi}
Set $\psi:=\psi_2\circ \cdots \circ \psi_N$ to be the injective map $\psi:\CSpec(A) \longrightarrow \CSpec(A')$.
We call $\psi$ the \emph{canonical embedding} of $\CSpec(A)$ into $\CSpec(A')$.
\end{definition}

\subsubsection{Partition of $\mathrm{C.Spec}(A)$}
Let $\mathcal{W}:=\mathbb{P}(\llbracket 1, N \rrbracket)$ denote the power set of $\llbracket 1, N \rrbracket$ and, for all $w\in \mathcal{W}$, set
\[\CSpec_w(A'):=\{Q\in \CSpec(A')\mid Q\cap \{T_1,\ldots T_N\}=\{T_i \mid i\in w\}\},\]
where $\{T_i\}_{i=1}^N$ are the generators of the quantum affine space $A'$. From \cite[Section 2.1]{GoodearlLetzter} we get:
\begin{lemma}\label{lemA'Partition}
The sets $\{\CSpec_w(A')\}_{w\in \mathcal{W}}$ provide a partition of $\CSpec(A')$.
\end{lemma}
We use $\psi$ to pull this partition back to one on $\CSpec(A)$. For each $w\in \mathcal{W}$ we define
\[\CSpec_w(A):=\psi^{-1}(\CSpec_w(A'))\]
and let $\mathcal{W}'\subseteq \mathcal{W}$ denote the set of all $w\in \mathcal{W}$ such that $\CSpec_w(A)\neq \emptyset$. (Note that $\mathcal{W}'$ depends on the expression of the algebra $A$ as an iterated Ore extension.) We immediately obtain a partition of $\CSpec(A)$.

\begin{theorem}\label{lemAPartition}
The set $\CSpec(A)$  has a partition indexed by the family $\mathcal{W}'$ so that,
\[\CSpec(A)=\bigsqcup_{w\in \mathcal{W}'} \CSpec_w(A), \quad \text{ where } ~|\mathcal{W}'|\leq |\mathcal{W}|=2^N.\]
\end{theorem}

\begin{definition}\label{defnCauchonDiagramW'}
	We call the partition $\{\CSpec_w(A)\}_{w\in \mathcal{W}'}$ the \emph{canonical partition} of $\CSpec(A)$, and we call each $w\in \mathcal{W}'$ a \emph{Cauchon diagram of $A$}.
\end{definition}

\subsection{Properties of the Canonical Embedding}\label{sectionProperties}
In order to use the deleting derivations algorithm for the purpose of calculating the PI degree of quotient algebras, we need to be able to test whether a completely prime ideal of $A'$ lies in the image of the canonical embedding. 

\begin{lemma}\label{lemMembership1}
Fix some $j\in \llbracket 2, N \rrbracket$ and let $Q\in \mathrm{C.Spec}(A^{(j)})$. Then,
\[ Q\in \Ima(\psi_j) \iff [\text{Either } x_j=y_j \notin Q \text{ or } \ker(g_j)\subseteq Q]. \]
\end{lemma}
\begin{proof}
Apply Lemmas \ref{lemBiIncHomP0} and \ref{lemPsi1} to the cases $y_j\notin Q$ and $y_j\in Q$ respectively.
\end{proof}

We define injective maps $f_j:\CSpec(A^{(j+1)})\rightarrow \CSpec(A')$, for all $j \in \llbracket 1, N \rrbracket$, with $f_1:=\mathrm{id}_{\CSpec(A')}$, the identity on $\CSpec(A')$ and, for all $j \in \llbracket 2, N \rrbracket$, $f_j:=\psi_2\circ\cdots\circ \psi_j$, so that $f_N= \psi$.

\begin{proposition}\label{propMembership1}
Let $Q\in \CSpec(A')$. The following are equivalent:
\begin{itemize}
\item $Q\in \Ima(\psi)$.
\item For all $j \in \llbracket 2, N \rrbracket$ we have $Q\in \Ima(f_{j-1})$ and either $X_j^{(j)}=X_j^{(j+1)}\notin f_{j-1}^{-1}(Q)$ or $\ker(g_j)\subseteq f_{j-1}^{-1}(Q)$.
\end{itemize}
\end{proposition}
\begin{proof}
Let $Q\in \mathrm{C.Spec}(A')$. Suppose $Q\in \Ima(\psi)$. Then $Q=\psi(P)$ for some $P\in \mathrm{C.Spec}(A)$.  Since $\psi=f_{j-1}\circ\psi_j\circ\cdots \psi_N$ then $Q=f_{j-1}(P_j)$ for all $j \in \llbracket 2, N \rrbracket$, where $P_j=\psi_j\circ \cdots \circ \psi_N(P)$. Hence $Q\in \Ima(f_{j-1})$ for all $j \in \llbracket 2, N \rrbracket$. From this we see that $f_{j-1}^{-1}(Q)\in \Ima(\psi_j)$, for all $j \in \llbracket 2, N\rrbracket$, and we apply Lemma \ref{lemMembership1} to $f_{j-1}^{-1}(Q)$ to conclude that either $X_j^{(j)}=X_j^{(j+1)}\notin f_{j-1}^{-1}(Q)$ or $\ker(g_j)\subseteq f_{j-1}^{-1}(Q)$.

Now suppose the second statement holds. By Lemma \ref{lemMembership1}, $f_{j-1}^{-1}(Q)\in \Ima(\psi_j)$ for all $j \in \llbracket 2, N \rrbracket$. Let $P':=f_{N-1}^{-1}(Q)$ so that $P'\in \Ima(\psi_N)$. Then $P'=\psi_N(P)$ for some $P\in \mathrm{C.Spec}(A)$ and
\[Q=f_{N-1}(P')=f_{N-1}(\psi_N(P))=\psi(P)\in \Ima(\psi).\]
\end{proof}

We finish this subsection by giving a sufficient condition for a completely prime ideal in $A'$ to be in the image of the canonical embedding. These next two results are not used in this paper, however they are stated here for the interested reader (proofs can be found in \cite[Section 4]{Thesis}).
\begin{theorem}\cite[Theorem 4.17]{Thesis}
Let $w\in \mathcal{W}'$. Then $\psi(\CSpec_w(A))$ is a (non-empty) closed subset of $\CSpec_w(A')$ and the map $\psi$ induces (by restriction) a bi-increasing homeomorphism from $\CSpec_w(A)$ to $\psi(\CSpec_w(A))$.
\end{theorem}

\begin{proposition}\cite[Proposition 4.18]{Thesis}\label{propCPinImage}
Let $w\in \mathcal{W}',\, P\in \CSpec_w(A)$, and $Q\in \CSpec_w(A')$.  If $\psi(P)\subseteq Q$ then $Q\in \Ima(\psi)$.
\end{proposition}

\subsection{Completely Prime Quotients}\label{sectionCompletelyPrimeQuotients}
In this subsection we extend the algorithm in Section \ref{sectionDDA} to apply to completely prime quotient algebras of quantum nilpotent algebras. The results of the previous sections allow us to construct proofs in a similar way to those found in \cite[Section 5.3]{Cauchon}. This extended algorithm will allow us to pull certain irreducible representations on $A'/Q$ (clarified in Section \ref{sectionIrredRepDDA}) back to irreducible representations on $A/P$. 
 
\subsubsection{$A^{(j)}/P$ and $A^{(j+1)}/\psi_j(P)$}
We start by extending one step of the algorithm to completely prime quotients. For some $j\in \llbracket 2, N \rrbracket$, let $P\in \CSpec(A^{(j+1)})$ and $Q=\psi_j(P)\in \CSpec(A^{(j)})$ be its image under the canonical embedding. Set
\[B^{(j+1)}:=A^{(j+1)}/P, \quad B^{(j)}:=A^{(j)}/Q.\]
We denote by $\bar{x}_1,\ldots, \bar{x}_N\in B^{(j+1)}$ and $\bar{y}_1,\ldots, \bar{y}_N\in B^{(j)}$ the canonical images of the generators $x_1,\ldots, x_N\in A^{(j+1)}$ in $B^{(j+1)}$ and $y_1,\ldots, y_N\in A^{(j)}$ in $B^{(j)}$, respectively.

\begin{lemma}\label{lemujEq0}
Suppose $\bar{x}_j=0$. Then there exists an algebra isomorphism $B^{(j)} \rightarrow B^{(j+1)}$ sending $\bar{y}_i \mapsto \bar{x}_i$ for all $i\in \llbracket 1, N \rrbracket$.
\end{lemma}
\begin{proof}
Since $x_j\in P$ then $Q=\psi_j^1(P)=g_j^{-1}(P/\langle x_j \rangle)$. Concatenating $g_j$ with the natural surjection $\pi:A^{(j+1)}/\langle x_j \rangle  \rightarrow A^{(j+1)}/P$ gives the following surjective algebra homomorphism:
\[\begin{array}{ccccc}
A^{(j)} &\overset{g_j}{\longrightarrow} & A^{(j+1)}/\langle x_j \rangle & \overset{\pi}{\longrightarrow} & A^{(j+1)}/P \\
y_i & \longmapsto & x_i+\langle x_j \rangle & \longmapsto & \bar{x}_i,
\end{array}\]
The desired isomorphism may then be constructed upon noting that $\ker(\pi \circ g_j)= g_j^{-1}(P/\langle x_j \rangle)=Q$.
\end{proof}

\begin{lemma}\label{lemujNotEq0}
	Suppose $\bar{x}_j\neq 0$ and let $Z_j:=\{\bar{x}_j^n \mid n\in \N\}$. Then the following hold:
	\begin{enumerate}[(i)]
		\item $Z_j$ is a multiplicative set of regular elements of $B^{(j+1)}$, which satisfies the two-sided Ore condition.
		\item There exists an injective algebra homomorphism $\gamma:B^{(j)} \rightarrow B^{(j+1)} Z_j^{-1}$ defined on the generators of $B^{(j)}$ in the following way:
		\[ \gamma(\bar{y}_i) = \begin{cases} \bar{x}_i & \text{ if } i\geq j; \\
			\sum_{n=0}^{\infty}q_j^{\frac{n(n+1)}{2}}(q_j-1)^{-n}\lambda_{j,i}^{-n}\overline{d_{j,n}(x_i)}\bar{x}_j^{-n} & \text{ if } i<j,
		\end{cases} \]
		where $\overline{d_{j,n}(x_i)}$ denotes the canonical image of $d_{j,n}(x_i)$ in $B^{(j+1)}$.
		\item If we identify $B^{(j)}$ with its image $\gamma(B^{(j)}) \subseteq B^{(j+1)}Z_j^{-1}$ then $Z_j$ is a multiplicative set of regular elements in $B^{(j)}$ satisfying the two-sided Ore condition. Furthermore, $B^{(j)} Z_j^{-1}=B^{(j+1)} Z_j^{-1}$.
	\end{enumerate}
\end{lemma}
\begin{proof}
Since $x_j \notin P$ then $Q=\psi_j^0(P)=PS_j^{-1} \cap A^{(j)}$, where $S_j=\{x_j^n \mid n\in \N\}$ is the multiplicatively closed set of regular elements in $A^{(j+1)}$ and $A^{(j)}$ satisfying the two-sided Ore condition (Theorem \ref{mainthm}(\ref{mainthmSj})). Denote the subalgebra $A^{(j+1)}S_j^{-1} = A^{(j)}S_j^{-1} \subseteq F$ by $\Omega$ and the completely prime ideal $PS_j^{-1}=QS_j^{-1}$ by $\Theta\lhd \Omega$. Note that $\Theta \cap A^{(j+1)}=P$ and $\Theta \cap A^{(j)}=Q$.

We define an injective algebra homomorphism $\gamma\,' : B^{(j+1)} \longrightarrow \Omega/\Theta; ~ a+ P  \longmapsto  a1^{-1}  + \Theta$ and identify $B^{(j+1)}$ with its image. Since $Z_j=\{\bar{x}_j^n \mid n\in \N\}=\{x_j^n+P \mid n\in \N\}=S_j+P$, its image under $\gamma\,'$ becomes 
\[\gamma\,'(Z_j)=Z_j1^{-1}+\Theta= (S_j+P)1^{-1} + \Theta = S_j1^{-1} +  \Theta \subseteq B^{(j+1)} \subseteq \Omega/ \Theta.\]
Identifying $Z_j$ with its image $\gamma\,'(Z_j)$, we observe that all elements of the set $Z_j$ are invertible in $B^{(j+1)}$. We can therefore write any element $b\in \Omega/\Theta$ as 
\begin{align}\label{eqnbinOmegaTheta}
b=a_1s_1^{-1}+\Theta=s_2^{-1}a_2+\Theta,
\end{align}
where $a_1, a_2\in A^{(j+1)}$ and $s_1, s_2 \in S_j$. Let $b_1, b_2 \in B^{(j+1)}$ and $z_1, z_2 \in Z_j$ such that
\[b_1=a_1 1^{-1}+\Theta, ~~~ b_2=a_2 1^{-1}+\Theta, ~~~ z_1=s_1 1^{-1}+\Theta, ~~~ z_2=s_2 1^{-1}+\Theta.\]
Using (\ref{eqnbinOmegaTheta}), we see that, for all $b\in \Omega/\Theta$,
\[ b= b_1z_1^{-1} = z_2^{-1}b_2. \]
This shows that the set $Z_j\subset B^{(j+1)}$ satisfies the two-sided Ore condition, thus proving property (i) of the lemma. We have also proved the equality $B^{(j+1)}Z_j^{-1}=\Omega/\Theta$, i.e. $(A^{(j+1)}/P)Z_j^{-1} = A^{(j+1)}S_j^{-1}/PS_j^{-1}$.

For part (ii) we use the fact that $\Theta \cap A^{(j)} = Q$ and $B^{(j+1)}Z_j^{-1}=\Omega/\Theta$ to define an injective homomorphism
\begin{align*}
\gamma: B^{(j)} ~~ &{} ~ \longrightarrow ~\Omega/\Theta\\
a + Q &{} ~ \longmapsto ~ a1^{-1}  + \Theta.
\end{align*}
Recall that the generators $y_i\in A^{(j)}$ are defined as:
\[ y_i :=  \begin{cases} x_i & \text{ if } i\geq j; \\
								\sum_{n=0}^{\infty}q_j^{\frac{n(n+1)}{2}}(q_j-1)^{-n}\lambda_{j,i}^{-n}d_{j,n}(x_i)x_j^{-n} & \text{ if } i<j.
				\end{cases} \]
It is straightforward to check that $\gamma(\bar{y}_i)=y_i 1^{-1}+\Theta$ gives the desired results for $i<j$ and $i \geq j$.

Part (iii) is proved in the same way as (i) by identifying $B^{(j)}$ with its image $\gamma(B^{(j)})$.
\end{proof}

Lemmas \ref{lemujEq0} and \ref{lemujNotEq0} imply:
\begin{lemma}\label{lemFrac(U)=Frac(V)}
	$\Frac(B^{(j+1)})=\Frac(B^{(j)})$.
\end{lemma}

\subsubsection{$A/P$ and $A'/\psi(P)$}
We continue to extend the algorithm to apply to completely prime quotients $A/P$. Let $P \in \CSpec(A)$ and $Q=\psi(P)\in \CSpec(A')$ and set the following notation:
\begin{itemize}
	\item $B:=A/P$ and set $\bar{X}_1, \ldots, \bar{X}_N\in B$ to be the canonical images of $X_1, \ldots, X_N\in A$.
	\item $B':=A'/Q$ and set $t_1,\ldots, t_N\in B'$ to be the canonical images of $T_1, \ldots, T_N\in A'$.
	\item For $j \in \llbracket 2, N+1 \rrbracket$, denote by $P_j:=\psi_j \circ \cdots \circ \psi_N(P)\in \CSpec(A^{(j)})$ the image of $P$ after $N-j+1$ steps of the deleting derivations algorithm. 
	\item For each $j \in \llbracket 2, N+1 \rrbracket$, define the algebra $B^{(j)}:=A^{(j)}/P_j$ and denote by $\bar{X}_1^{(j)}, \ldots, \bar{X}_N^{(j)}$ the canonical images of $X_1^{(j)}, \ldots, X_N^{(j)}$ in $B^{(j)}$. Note: $B^{(N+1)}=B$ with $(\bar{X}_1^{(N+1)},\ldots, \bar{X}_N^{(N+1)})=(\bar{X}_1, \ldots, \bar{X}_N)$, and $B^{(2)}=B'$ with $(\bar{X}_1^{(2)},\ldots, \bar{X}_N^{(2)})=(t_1, \ldots, t_N)$.
\end{itemize}

\begin{proposition}\label{propUjSubalgFracU}
For each $j \in \llbracket 2, N+1 \rrbracket$, $B^{(j)}$ is a subalgebra of $\Frac(B)$ generated by $\bar{X}_1^{(j)},\ldots, \bar{X}_N^{(j)}$ and there is an algebra homomorphism,
\begin{align*}
	f_j: ~~ A^{(j)} &{} ~ \longrightarrow ~ \Frac(B) \\
	X_i^{(j)} &{} ~ \longmapsto ~ \bar{X}_i^{(j)}
\end{align*}
with image $B^{(j)}$ and kernel $P_j$.
\end{proposition}
\begin{proof}
Since $B^{(j)}$ is (trivially) a subalgebra of $\text{Frac}(B^{(j)})$ generated by $\bar{X}_1^{(j)}, \ldots, \bar{X}_N^{(j)}$, Lemma \ref{lemFrac(U)=Frac(V)} implies that $B^{(j)}$ is a subalgebra of $\Frac(B)$ generated by these same elements. The homomorphism $f_j$ is the concatenation of the natural embedding $B^{(j)} \hookrightarrow \text{Frac}(B)$ with the canonical surjection, $\pi_j:A^{(j)} \rightarrow A^{(j)}/P_j=B^{(j)}$. The stated image and kernel are easily verified.
\end{proof}

\begin{proposition}\label{propconfusedalgs2}
	Let $j \in \llbracket 2, N+1 \rrbracket$.
	\begin{enumerate}[(i)]
		\item If $\bar{X}_j^{(j+1)}=0$ then $\bar{X}_i^{(j)}=\bar{X}_i^{(j+1)}$ for all $i \in \llbracket 1, N \rrbracket$.
		\item Suppose $\bar{X}_j^{(j+1)}\neq 0$ and set $(x_1,\ldots, x_N):=(X_1^{(j+1)},\ldots,X_N^{(j+1)})$.  Then the generators of $B^{(j)}$ are can be obtained as follows:
		\[ \bar{X}_i^{(j)}= \begin{cases} \bar{X}_i^{(j+1)} & \text{ if } i\geq j; \\
			\sum_{n=0}^{\infty}q_j^{\frac{n(n+1)}{2}}(q_j-1)^{-n}\lambda_{j,i}^{-n}f_{j+1} \circ d_{j,n}(x_i)(\bar{X}_j^{(j+1)})^{-n} & \text{ if } i<j,
		\end{cases} \]
		where $f_{j+1}$ is the map defined in Proposition \ref{propUjSubalgFracU}.
		\item Suppose $\bar{X}_j^{(j+1)}\neq 0$ and let $Z_j=\{(\bar{X}_j^{(j+1)})^n \mid n\in \N \}=\{(\bar{X}_j^{(j)})^n \mid n\in \N \}$ be a multiplicatively closed set of regular elements in $B^{(j)}$ and $B^{(j+1)}$. Then $Z_j$ satisfies the two-sided Ore condition in both $B^{(j)}$ and $B^{(j+1)}$ and we have $B^{(j)}Z_j^{-1}=B^{(j+1)}Z_j^{-1}$.
	\end{enumerate}
\end{proposition}
\begin{proof}
Lemmas \ref{lemujEq0} and \ref{lemFrac(U)=Frac(V)} prove part (i).  Part (ii) follows from Lemma \ref{lemujNotEq0} once one notes that $f_{j+1} \circ d_{j,n}(x_i)=\overline{d_{j,n}(x_i)}\cdot 1^{-1} \in \text{Frac}(B)$. Part (iii) also follows directly from Lemma \ref{lemujNotEq0}.
\end{proof}

Let $w\in \mathcal{W}'$ with $P\in \mathrm{C.Spec}_w(A)$ and $Q=\psi(P)$. By the definition of $\mathrm{C.Spec}_w(A')$, we have that $T_i\in Q$ if and only if $i\in w$ or, equivalently, $t_i=0$ if and only if $i \in w$.  Let $i\in \bar{w}:=\{1,\ldots, N\}\backslash w$ so that $t_i\neq 0$. Then, since $T_i$ is normal in $A'$ and $Q$ is completely prime, $t_i$ is normal and regular in $B'$, hence it is invertible in $\text{Frac}(B')=\text{Frac}(B)$. We denote by $\Sigma$ the multiplicatively closed set of regular elements in $B'$ generated by all $t_i$ such that $i\in \bar{w}$. From this set we define, recursively, the sets $\Sigma_j\subset B^{(j)}$ for $j \in \llbracket 2, N+1 \rrbracket$ in the following way:
\[ \Sigma_2:= \Sigma, \quad \Sigma_{j+1}:=B^{(j+1)}\cap \Sigma_j. \]

The next result extends Proposition \ref{propconfusedalgs} to quotient algebras.
\begin{proposition}\label{propconfusedalgs3}
	For each $j\in \llbracket 2, N+1 \rrbracket$ the following statements hold:
	\begin{enumerate}[(i)]
		\item $\Sigma_j$ is a multiplicatively closed set of regular elements in $B^{(j)}$ which contains, as a subset, $\{\bar{X}_i^{(j)}\mid i\in \llbracket  j-1, N  \rrbracket \text{ and } \bar{X}_i^{(j)}\neq 0 \}$;
		\item $\Sigma_j$ satisfies the two-sided Ore condition in $B^{(j)}$;
		\item The algebras $B^{(j)}\Sigma_j^{-1}$, when considered as subalgebras of $\Frac(B)$, are all equal.
	\end{enumerate}
\end{proposition}
\begin{proof}
We proceed by induction on $j$. When $j=2$, statements (i) and (ii) are immediately satisfied by the discussion preceding this proposition. Fix $j\in \llbracket 2, N \rrbracket$ and assume statements (i) and (ii) are true. We show that these properties are also true when replacing $j$ with $j+1$ and that $B^{(j)}\Sigma_j^{-1}=B^{(j+1)}\Sigma_{j+1}^{-1}$.  We consider two cases: $\bar{X}_j^{(j+1)}=0$ and $\bar{X}_j^{(j+1)}\neq 0$.

When $\bar{X}_j^{(j+1)}=0$ we apply Proposition \ref{propconfusedalgs2}(i) to obtain $\bar{X}_i^{(j)}=\bar{X}_i^{(j+1)}$ for all $i\in \llbracket 1, N \rrbracket$. Therefore $B^{(j)}=B^{(j+1)}$, and statements (i) and (ii) follow immediately by the inductive hypothesis and the fact that $\Sigma_{j+1}\subseteq \Sigma_j$.

Now suppose $\bar{X}_j^{(j+1)}\neq 0$. Applying Proposition \ref{propconfusedalgs2}(ii) gives $\bar{X}_i^{(j)}=\bar{X}_i^{(j+1)}$ for all $i \geq j$ and, by the induction hypothesis, we know that
\[ \{\bar{X}_i^{(j)} \mid i \in \llbracket j-1, N \rrbracket, \; \bar{X}_i^{(j)} \neq 0 \} ~\subseteq~ \Sigma_j. \]
Therefore,
\begin{align*}
\Sigma_{j+1} := B^{(j+1)} \cap  \Sigma_j &{} ~\supseteq~ B^{(j+1)} \cap \{\bar{X}_i^{(j)} \mid i\in \llbracket j-1, N \rrbracket, \; \bar{X}_i^{(j)}\neq 0 \} \\
&{} ~=~ \{ \bar{X}_i^{(j+1)} \mid i \in \llbracket j, N \rrbracket, \; \bar{X}_i^{(j+1)}\neq 0\}.
\end{align*}
The set $\Sigma_j$ is multiplicatively closed, by the induction hypothesis, which means $\Sigma_{j+1}$ is multiplicatively closed and it contains regular elements because $P_{j+1}$ is a completely prime ideal. Hence $B^{(j+1)}$ is a domain in which all nonzero elements are regular. 
This proves part (i).

Note that $Z_j\subset \Sigma_{j+1}$ and $Z_j \subset \Sigma_j$.  Applying Proposition \ref{propconfusedalgs2}(iii), we see that
\[B^{(j+1)} \subset B^{(j+1)}Z_j^{-1}=B^{(j)}Z_j^{-1}\subset B^{(j)}\Sigma_j^{-1}. \]
These inclusions mirror those found in (\ref{eqnconfusedalg}), and we may apply the rest of the method used in the proof of Proposition \ref{propconfusedalgs} to conclude parts (ii) and (iii) of this proposition.
\end{proof}

An immediate consequence of this proposition is that we can now show equivalence of the total rings of fractions of $A/P$ and $A'/\psi(P)$ in the case where $P\in \mathrm{C.Spec}_w(A)$, for some $w \in \mathcal{W}'$.

\begin{theorem}\label{thmDDAonQuotients}
	Let $w\in \mathcal{W}', \: P\in \CSpec_w(A)$, and $Q=\psi(P)\in \CSpec_w(A')$. Let $\Sigma$ be the multiplicatively closed set of elements in $A'/Q$ which is generated by all the generators $t_i$ such that $i\in \bar{w}=\{1,\ldots, N\} \backslash w$, where $t_i$ is the canonical image of $T_i$ in $A'/Q$. Then,
	\begin{enumerate}[(i)]
		\item $\Sigma$ is a multiplicatively closed set of regular, normal elements in $A'/Q$. 
		
		\item There exists a multiplicatively closed set of regular elements, $\Gamma$, in $A/P$ which satisfies the two-sided Ore condition in $A/P$ and is such that $(A/P)\Gamma^{-1} = (A'/Q)\Sigma^{-1}$.
		\item $\Frac(A/P)=\Frac(A'/Q)$. Furthermore, $A/P$ and $A'/Q$ are PI algebras if all $\lambda_{i,j}$ with $i, j \in \bar{w}$ are roots of unity and, in this case, $\text{PI-deg}(A/P)=\text{PI-deg}(A'/Q)$.\label{thmDDAonQuotients.iii}
	\end{enumerate}
\end{theorem}

The following ideals are defined for later use as they ensure that at the end of the deleting derivations algorithm on $A/P$ we obtain a quantum affine space $A'/\psi(P)$.
\begin{definition}\label{defnCauchonIdeal}
Let $w\in \mathcal{W}'$. We call $P_w\in \CSpec_w(A)$ a \emph{Cauchon ideal} if  $\psi(P_w)=J_w := \langle T_i \in A' \mid i \in w \rangle \in \CSpec_w(A')$.
\end{definition}

\begin{remark}
It can be shown that a rational torus action on $A$ induces a rational torus action on $A'$ and that $\psi$ sends torus-invariant completely prime ideals in $A$ to torus-invariant completely prime ideals in $A'$. This matches analogous results in the generic setting \cite{Cauchon} and the Poisson setting \cite{LaunoisLecoutre}. Unlike in those settings, this observation does not immediately yield results allowing us to determine explicitly what $W'$ is for specific algebras $A$. For this reason, we omit these results from this paper.
\end{remark}

\section{PI degree of completely prime quotients of quantum matrices at roots of unity}\label{sectionPIdegCPQofQMs}

In this section, we consider iterated Ore extensions $A=\K[X_1][X_2;\sigma_2,\delta_2]\ldots [X_N;\sigma_N,\delta_N]$ satisfying Hypothesis \ref{hyp1} {\bf at roots of unity}. That is, we assume that all the parameters $\lambda_{i,j}$ from Hypothesis \ref{hyp1} are roots of unity. 

In that case, it follows from \cite[Theorem 1.2]{Haynal} that the algebra $A$ is PI and that, with the notation of the previous section: 
$$\PI(A) =\PI(A') .$$

Our aim in this section is to provide techniques to compute the PI degree of completely prime quotients of $A$. In particular, we will focus on the case where $A$ is a single parameter quantum matrix algebra (at roots of unity).

\subsection{PI degree of completely prime quotients of $A$}

Under our assumptions, $A$ is a PI algebra and this implies that every completely prime quotient of $A$ is a PI algebra too. Our next aim is to compute the PI degree and construct irreducible representations of maximal dimension of quotients by Cauchon ideals.

\subsection{PI degree of quotients by a Cauchon ideal}\label{sectionPIdegCauchonIdeal}

In the case where $P_w$ is the Cauchon ideal associated to a Cauchon diagram $w \in \mathcal W'$, we get that $\psi(P_w) = \langle T_i ~|~ i \in w\rangle$ and so $A'/Q$ is a quantum affine space. Namely, $A'/Q=\oh_{{\bf \Lambda_{w}}}(\K^{N-|w|})$, where ${\bf \Lambda_{w}}$ is the matrix deduced from ${\bf \Lambda}=(\lambda_{ij})$ by removing rows and columns indexed by elements of $w$.

 In view of Theorem~\ref{thmDDAonQuotients}(\ref{thmDDAonQuotients.iii}) and the above discussion, we obtain the following result.
 
 \begin{theorem}
 \label{PI2}
 Let $P_w$ be the Cauchon ideal associated to a Cauchon diagram $w \in \mathcal W'$. Then $\PI(A/P_w)=\PI(\oh_{{\bf \Lambda_{w}}}(\K^{N-|w|}))$.
 \end{theorem}

De Concini and Procesi developed techniques to compute the PI degree of a quantum affine space. More precisely, they prove the following result.

\begin{theorem}\cite[Proposition 7.1]{DeConciniProcesi}
Let $S= (s_{ij})$ be a multiplicatively skew-symmetric $n\times n$ matrix with coefficients in $\K$. Then:
\begin{enumerate}
\item If all $s_{ij}$ are roots of unity, then there exists a primitive root of unity $q\in \K^{\times}$ and integers $a_{ij}$ such that $s_{ij}=q^{a_{ij}}$ for all $i,j$.
\item Assume $s_{ij}=q^{a_{ij}}$ for all $i,j$, where $q$ is a primitive $\ell^{\mbox{th}}$ root of unity and $a_{ij}\in \Z$ for all $i,j$. Set $M:=(a_{ij})\in M_{n}(\Z)$; this is a skew-symmetric matrix. Then the PI degree of the corresponding quantum affine space $\oh_{q^M} (\K^n):= \oh_{S}(\K^{n})$ is $\sqrt{h}$, where $h$ is the cardinality of the image of the homomorphism
\[\begin{tikzcd}
 \Z^{n} \arrow{r}{M} & \Z^{n} \arrow{r}{\pi} & (\Z/\ell\Z)^n,
\end{tikzcd}\]
where $\pi$ denotes the canonical epimorphism. 
\end{enumerate} 
\end{theorem}
 
 The above result together with Theorem~\ref{PI2} allows us to compute the PI degree of quotients by Cauchon ideals. 
 
 In the following section, we illustrate this in the case when $A$ is a single parameter quantum matrix algebra. 
 
 \subsection{PI degree of quotients of single parameter quantum matrix algebras by Cauchon ideals}
  
 In this section, we assume that $\lambda_{ij}=q^{m_{ij}}$ for all $i,j$, where $q$ is a primitive $\ell^{\mbox{th}}$ root of unity and $m_{ij}\in \Z$ for all $i,j$. In this case, we say that $A$ is {\em uniparameter}. We set $M:=(m_{ij}) \in M_N(\K)$, so that $M$ is a skew-symmetric integral matrix. It is a well-known result that every skew-symmetric integral matrix $S$ is congruent (in the sense of \cite[Chapter IV]{Newman}) to its skew-normal form, that is, to a block diagonal matrix of the form
\[ S=\left(\begin{smallmatrix}
 0 & h_1 & & & & & & & \\
 -h_1 & 0 & & & & & & & \\
 & & 0 & h_2 & & & & & \\
& & -h_2 & 0 & & & & & \\
 & & & & \ddots & & & \\
 & & & & & 0 & h_s & \\
 & & & & & -h_s & 0 & \\
 & & & & & & & \mathbf{0} 
\end{smallmatrix}\right), \]
where $\mathbf{0}$ is a square matrix of zeros of dimension $\dim(\ker(S))$ so that $2s=N-\dim(\ker(S))$, and $h_1,\, h_1,\, h_2,\, h_2,\, \ldots,\, h_s,\, h_s \in \Z\backslash \{0\}$ are called the \emph{invariant factors} of $M$. As they always come in pairs, from now on we will avoid repetition and list the invariant factors simply as $h_1,\, h_2,\,  \ldots,\, h_s$. These have the property that $h_i|h_{i+1}$ for all $i\in \llbracket 1, s-1 \rrbracket$.

\begin{theorem}\label{thmPIdegAgain}
	Assume $A$ is a uniparameter iterated Ore extension satisfying Hypothesis \ref{hyp1} with parameter $q$ being a primitive $\ell^{\mbox{th}}$ root of unity. Let $w\in \mathcal{W}'$ and $P_w:=\psi^{-1}(J_w)$ be the corresponding Cauchon ideal. Then 
	$$\PI(A/P_{w})=\prod_{i=1}^{\frac{N-\dim(\ker(M(w)))}{2}} \frac{\ell}{\gcd(h_i, \ell) },$$
where the $h_i$ are the invariant factors of the matrix $M(w)$ deduced from $M$ by removing rows and columns indexed by $w$.
		\end{theorem}
		
We finish this section by studying the case of single parameter quantum matrices. That is, we assume that $A= \oh_{q}(M_{m,n}(\K))$ with $q$ being a primitive $\ell^{\mbox{th}}$ root of unity. Recall that the quantised coordinate ring of $m\times n$ matrices over $\K$ denoted by $\oh_q(M_{m,n}(\K))$ is the $\K$-algebra generated by $m\times n$ indeterminates $\{X_{i,j}\}_{i,j=1}^{m,n}$ subject to the following relations: for $(1,1) \leq  (i,j) < (s,t) \leq (m,n)$ (in lexicographic ordering), we have
\begin{align*}
X_{i,j}X_{s,t}= \begin{cases} 
							X_{s,t}X_{i,j} & i<s, \, j>t; \\
							qX_{s,t}X_{i,j} & (i=s, \, j<t) \text{ or } (i<s, \, j=t); \\
							X_{s,t}X_{i,j}+(q-q^{-1})X_{i,t}X_{s,j} & i<s, \, j<t.							
				 \end{cases}
\end{align*}

Under our assumption that $q$ is a primitive $\ell^{\mbox{th}}$ root of unity, the algebra $A= \oh_{q}(M_{m,n}(\K))$  is a  uniparameter iterated Ore extension (with the generators added in the lexicographic order) satisfying Hypothesis \ref{hyp1}, see \cite[Section 5.3]{Haynal} for details.  

 Let $w$ be a Cauchon diagram and $P_w$ the corresponding Cauchon ideal in $A$. It follows from the above result that to compute the PI degree of $A/P_w$, we need to compute the dimension of $\ker (M(w))$ and the invariant factors of $M(w))$. This was done in \cite{BellCasteelsLaunois} and \cite{BellLaunoisRogers} respectively. Before stating our main result, we recall necessary results from these two articles.

Given any diagram $w \in \mathcal W$, we associate an $m \times n$ diagram $D(w)$, where $D(w)$ is the $m\times n$ grid whose square in position $(i,j)$ is coloured black if $(i,j) \in w$ and white if $(i,j) \notin w$.   We may compute its \emph{toric permutation} $\tau = \tau_w$, as defined in \cite[Section 4.1]{BellCasteelsLaunois}, by laying pipes over the squares such that we place a ``cross" on each black square and a ``hyperbola" on each white square. We label the sides of the diagram with the numbers $1,\ldots, m+n$ such that each pair of opposite sides share the same labels in the same order. The permutation, $\tau_w$, may then be read off this diagram by defining $\tau_w(i)$ to be the label (on the left or top side of $D(w)$) reached by following the pipe starting at label $i$ (on the right or bottom side of $D(w)$). See Figure \ref{FigToricPerm} for an example of a diagram with $\tau=(17)(26384)$.

\definecolor{light-gray}{gray}{0.6}
\begin{figure}[h]
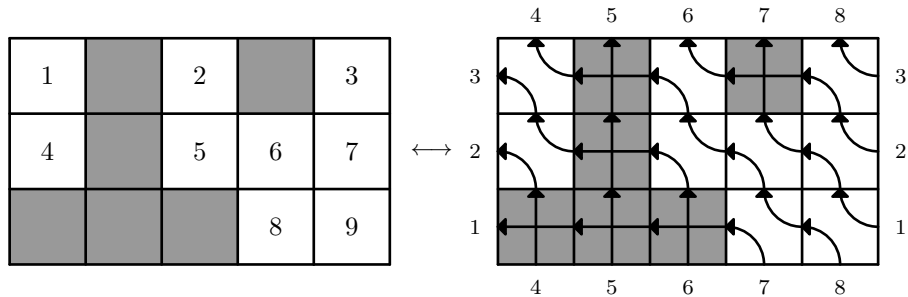

\begin{center}
\begin{tabular}{cc}
\begin{pgfpicture}{0cm}{0cm}{6cm}{4cm}%
\pgfsetroundjoin \pgfsetroundcap%
\pgfsetfillcolor{light-gray}
\pgfmoveto{\pgfxy(0.5,0.5)}\pgflineto{\pgfxy(0.5,1.5)}\pgflineto{\pgfxy(3.5,1.5)}\pgflineto{\pgfxy(3.5,0.5)}\pgflineto{\pgfxy(0.5,0.5)}\pgffill
\pgfmoveto{\pgfxy(1.5,1.5)}\pgflineto{\pgfxy(1.5,3.5)}\pgflineto{\pgfxy(2.5,3.5)}\pgflineto{\pgfxy(2.5,1.5)}\pgflineto{\pgfxy(1.5,1.5)}\pgffill
\pgfmoveto{\pgfxy(3.5,2.5)}\pgflineto{\pgfxy(3.5,3.5)}\pgflineto{\pgfxy(4.5,3.5)}\pgflineto{\pgfxy(4.5,2.5)}\pgflineto{\pgfxy(3.5,2.5)}\pgffill
\pgfsetlinewidth{1pt}
\pgfxyline(0.5,0.5)(4.5,0.5)
\pgfxyline(0.5,1.5)(4.5,1.5)
\pgfxyline(0.5,2.5)(4.5,2.5)
\pgfxyline(0.5,3.5)(4.5,3.5)
\pgfxyline(0.5,0.5)(0.5,3.5)
\pgfxyline(1.5,0.5)(1.5,3.5)
\pgfxyline(2.5,0.5)(2.5,3.5)
\pgfxyline(3.5,0.5)(3.5,3.5)
\pgfxyline(4.5,0.5)(4.5,3.5)
\pgfxyline(5.5,0.5)(5.5,3.5)
\pgfxyline(4.5,0.5)(5.5,0.5)
\pgfxyline(4.5,1.5)(5.5,1.5)
\pgfxyline(4.5,2.5)(5.5,2.5)
\pgfxyline(4.5,3.5)(5.5,3.5)

\pgfputat{\pgfxy(1,3)}{\pgfnode{rectangle}{center}{\color{black} $1$}{}{\pgfusepath{}}}
\pgfputat{\pgfxy(3,3)}{\pgfnode{rectangle}{center}{\color{black} $2$}{}{\pgfusepath{}}}
\pgfputat{\pgfxy(5,3)}{\pgfnode{rectangle}{center}{\color{black} $3$}{}{\pgfusepath{}}}
\pgfputat{\pgfxy(1,2)}{\pgfnode{rectangle}{center}{\color{black} $4$}{}{\pgfusepath{}}}
\pgfputat{\pgfxy(3,2)}{\pgfnode{rectangle}{center}{\color{black} $5$}{}{\pgfusepath{}}}
\pgfputat{\pgfxy(4,2)}{\pgfnode{rectangle}{center}{\color{black} $6$}{}{\pgfusepath{}}}
\pgfputat{\pgfxy(5,2)}{\pgfnode{rectangle}{center}{\color{black} $7$}{}{\pgfusepath{}}}
\pgfputat{\pgfxy(4,1)}{\pgfnode{rectangle}{center}{\color{black} $8$}{}{\pgfusepath{}}}
\pgfputat{\pgfxy(5,1)}{\pgfnode{rectangle}{center}{\color{black} $9$}{}{\pgfusepath{}}}

\pgfputat{\pgfxy(6,2)}{\pgfnode{rectangle}{center}{\color{black}\footnotesize $  ~~~~\longleftrightarrow~~~$}{}{\pgfusepath{}}}
\end{pgfpicture}

 &

\begin{pgfpicture}{0cm}{0cm}{6cm}{4cm}%
\pgfsetroundjoin \pgfsetroundcap%
\pgfsetfillcolor{light-gray}
\pgfmoveto{\pgfxy(0.5,0.5)}\pgflineto{\pgfxy(0.5,1.5)}\pgflineto{\pgfxy(3.5,1.5)}\pgflineto{\pgfxy(3.5,0.5)}\pgflineto{\pgfxy(0.5,0.5)}\pgffill
\pgfmoveto{\pgfxy(1.5,1.5)}\pgflineto{\pgfxy(1.5,3.5)}\pgflineto{\pgfxy(2.5,3.5)}\pgflineto{\pgfxy(2.5,1.5)}\pgflineto{\pgfxy(1.5,1.5)}\pgffill
\pgfmoveto{\pgfxy(3.5,2.5)}\pgflineto{\pgfxy(3.5,3.5)}\pgflineto{\pgfxy(4.5,3.5)}\pgflineto{\pgfxy(4.5,2.5)}\pgflineto{\pgfxy(3.5,2.5)}\pgffill
\pgfsetlinewidth{1pt}
\pgfxyline(0.5,0.5)(4.5,0.5)
\pgfxyline(0.5,1.5)(4.5,1.5)
\pgfxyline(0.5,2.5)(4.5,2.5)
\pgfxyline(0.5,3.5)(4.5,3.5)
\pgfxyline(0.5,0.5)(0.5,3.5)
\pgfxyline(1.5,0.5)(1.5,3.5)
\pgfxyline(2.5,0.5)(2.5,3.5)
\pgfxyline(3.5,0.5)(3.5,3.5)
\pgfxyline(4.5,0.5)(4.5,3.5)
\pgfxyline(5.5,0.5)(5.5,3.5)
\pgfxyline(4.5,0.5)(5.5,0.5)
\pgfxyline(4.5,1.5)(5.5,1.5)
\pgfxyline(4.5,2.5)(5.5,2.5)
\pgfxyline(4.5,3.5)(5.5,3.5)

\pgfsetlinewidth{1pt}
\color{black}

\pgfmoveto{\pgfxy(0.5,3)}\pgfpatharc{90}{0}{0.5cm}
\pgfstroke
\pgfmoveto{\pgfxy(0.5,3)}\pgflineto{\pgfxy(0.6,3.1)}\pgflineto{\pgfxy(0.6,2.9)}\pgflineto{\pgfxy(0.5,3)}\pgfclosepath\pgffillstroke

\pgfmoveto{\pgfxy(1,3.5)}\pgfpatharc{180}{270}{0.5cm}
\pgfstroke
\pgfmoveto{\pgfxy(1,3.5)}\pgflineto{\pgfxy(0.9,3.4)}\pgflineto{\pgfxy(1.1,3.4)}\pgflineto{\pgfxy(1,3.5)}\pgfclosepath\pgffillstroke

\color{black}
\pgfxyline(2,2.5)(2,3.5)
\pgfmoveto{\pgfxy(1.9,3.4)}\pgflineto{\pgfxy(2,3.5)}\pgflineto{\pgfxy(2.1,3.4)}\pgflineto{\pgfxy(1.9,3.4)}\pgfclosepath\pgffillstroke

\pgfxyline(1.5,3)(2.5,3)
\pgfmoveto{\pgfxy(1.5,3)}\pgflineto{\pgfxy(1.6,3.1)}\pgflineto{\pgfxy(1.6,2.9)}\pgflineto{\pgfxy(1.5,3)}\pgfclosepath\pgffillstroke

\pgfmoveto{\pgfxy(2.5,3)}\pgfpatharc{90}{0}{0.5cm}
\pgfstroke
\pgfmoveto{\pgfxy(2.5,3)}\pgflineto{\pgfxy(2.6,3.1)}\pgflineto{\pgfxy(2.6,2.9)}\pgflineto{\pgfxy(2.5,3)}\pgfclosepath\pgffillstroke

\pgfmoveto{\pgfxy(3,3.5)}\pgfpatharc{180}{270}{0.5cm}
\pgfstroke
\pgfmoveto{\pgfxy(3,3.5)}\pgflineto{\pgfxy(2.9,3.4)}\pgflineto{\pgfxy(3.1,3.4)}\pgflineto{\pgfxy(3,3.5)}\pgfclosepath\pgffillstroke


\pgfxyline(4,2.5)(4,3.5)
\pgfmoveto{\pgfxy(3.9,3.4)}\pgflineto{\pgfxy(4,3.5)}\pgflineto{\pgfxy(4.1,3.4)}\pgflineto{\pgfxy(3.9,3.4)}\pgfclosepath\pgffillstroke

\pgfxyline(3.5,3)(4.5,3)
\pgfmoveto{\pgfxy(3.5,3)}\pgflineto{\pgfxy(3.6,3.1)}\pgflineto{\pgfxy(3.6,2.9)}\pgflineto{\pgfxy(3.5,3)}\pgfclosepath\pgffillstroke


\pgfmoveto{\pgfxy(0.5,2)}\pgfpatharc{90}{0}{0.5cm}
\pgfstroke
\pgfmoveto{\pgfxy(0.5,2)}\pgflineto{\pgfxy(0.6,2.1)}\pgflineto{\pgfxy(0.6,1.9)}\pgflineto{\pgfxy(0.5,2)}\pgfclosepath\pgffillstroke

\pgfmoveto{\pgfxy(1,2.5)}\pgfpatharc{180}{270}{0.5cm}
\pgfstroke
\pgfmoveto{\pgfxy(1,2.5)}\pgflineto{\pgfxy(0.9,2.4)}\pgflineto{\pgfxy(1.1,2.4)}\pgflineto{\pgfxy(1,2.5)}\pgfclosepath\pgffillstroke


\pgfxyline(2,1.5)(2,2.5)
\pgfmoveto{\pgfxy(1.9,2.4)}\pgflineto{\pgfxy(2,2.5)}\pgflineto{\pgfxy(2.1,2.4)}\pgflineto{\pgfxy(1.9,2.4)}\pgfclosepath\pgffillstroke

\pgfxyline(1.5,2)(2.5,2)
\pgfmoveto{\pgfxy(1.5,2)}\pgflineto{\pgfxy(1.6,2.1)}\pgflineto{\pgfxy(1.6,1.9)}\pgflineto{\pgfxy(1.5,2)}\pgfclosepath\pgffillstroke

\pgfmoveto{\pgfxy(2.5,2)}\pgfpatharc{90}{0}{0.5cm}
\pgfstroke
\pgfmoveto{\pgfxy(2.5,2)}\pgflineto{\pgfxy(2.6,2.1)}\pgflineto{\pgfxy(2.6,1.9)}\pgflineto{\pgfxy(2.5,2)}\pgfclosepath\pgffillstroke

\pgfmoveto{\pgfxy(3,2.5)}\pgfpatharc{180}{270}{0.5cm}
\pgfstroke
\pgfmoveto{\pgfxy(3,2.5)}\pgflineto{\pgfxy(2.9,2.4)}\pgflineto{\pgfxy(3.1,2.4)}\pgflineto{\pgfxy(3,2.5)}\pgfclosepath\pgffillstroke


\pgfmoveto{\pgfxy(3.5,2)}\pgfpatharc{90}{0}{0.5cm}
\pgfstroke
\pgfmoveto{\pgfxy(3.5,2)}\pgflineto{\pgfxy(3.6,2.1)}\pgflineto{\pgfxy(3.6,1.9)}\pgflineto{\pgfxy(3.5,2)}\pgfclosepath\pgffillstroke

\pgfmoveto{\pgfxy(4,2.5)}\pgfpatharc{180}{270}{0.5cm}
\pgfstroke
\pgfmoveto{\pgfxy(4,2.5)}\pgflineto{\pgfxy(3.9,2.4)}\pgflineto{\pgfxy(4.1,2.4)}\pgflineto{\pgfxy(4,2.5)}\pgfclosepath\pgffillstroke


\pgfxyline(1,0.5)(1,1.5)
\pgfmoveto{\pgfxy(0.9,1.4)}\pgflineto{\pgfxy(1,1.5)}\pgflineto{\pgfxy(1.1,1.4)}\pgflineto{\pgfxy(0.9,1.4)}\pgfclosepath\pgffillstroke

\pgfxyline(0.5,1)(1.5,1)
\pgfmoveto{\pgfxy(0.5,1)}\pgflineto{\pgfxy(0.6,1.1)}\pgflineto{\pgfxy(0.6,0.9)}\pgflineto{\pgfxy(0.5,1)}\pgfclosepath\pgffillstroke


\pgfxyline(2,0.5)(2,1.5)
\pgfmoveto{\pgfxy(1.9,1.4)}\pgflineto{\pgfxy(2,1.5)}\pgflineto{\pgfxy(2.1,1.4)}\pgflineto{\pgfxy(1.9,1.4)}\pgfclosepath\pgffillstroke

\pgfxyline(1.5,1)(2.5,1)
\pgfmoveto{\pgfxy(1.5,1)}\pgflineto{\pgfxy(1.6,1.1)}\pgflineto{\pgfxy(1.6,0.9)}\pgflineto{\pgfxy(1.5,1)}\pgfclosepath\pgffillstroke


\pgfxyline(3,0.5)(3,1.5)
\pgfmoveto{\pgfxy(2.9,1.4)}\pgflineto{\pgfxy(3,1.5)}\pgflineto{\pgfxy(3.1,1.4)}\pgflineto{\pgfxy(2.9,1.4)}\pgfclosepath\pgffillstroke

\pgfxyline(2.5,1)(3.5,1)
\pgfmoveto{\pgfxy(2.5,1)}\pgflineto{\pgfxy(2.6,1.1)}\pgflineto{\pgfxy(2.6,0.9)}\pgflineto{\pgfxy(2.5,1)}\pgfclosepath\pgffillstroke

\pgfmoveto{\pgfxy(3.5,1)}\pgfpatharc{90}{0}{0.5cm}
\pgfstroke
\pgfmoveto{\pgfxy(3.5,1)}\pgflineto{\pgfxy(3.6,1.1)}\pgflineto{\pgfxy(3.6,0.9)}\pgflineto{\pgfxy(3.5,1)}\pgfclosepath\pgffillstroke

\pgfmoveto{\pgfxy(4,1.5)}\pgfpatharc{180}{270}{0.5cm}
\pgfstroke
\pgfmoveto{\pgfxy(4,1.5)}\pgflineto{\pgfxy(3.9,1.4)}\pgflineto{\pgfxy(4.1,1.4)}\pgflineto{\pgfxy(4,1.5)}\pgfclosepath\pgffillstroke

\pgfmoveto{\pgfxy(4.5,3)}\pgfpatharc{90}{0}{0.5cm}
\pgfstroke
\pgfmoveto{\pgfxy(4.5,3)}\pgflineto{\pgfxy(4.6,3.1)}\pgflineto{\pgfxy(4.6,2.9)}\pgflineto{\pgfxy(4.5,3)}\pgfclosepath\pgffillstroke

\pgfmoveto{\pgfxy(5,3.5)}\pgfpatharc{180}{270}{0.5cm}
\pgfstroke
\pgfmoveto{\pgfxy(5,3.5)}\pgflineto{\pgfxy(4.9,3.4)}\pgflineto{\pgfxy(5.1,3.4)}\pgflineto{\pgfxy(5,3.5)}\pgfclosepath\pgffillstroke


\pgfmoveto{\pgfxy(4.5,2)}\pgfpatharc{90}{0}{0.5cm}
\pgfstroke
\pgfmoveto{\pgfxy(4.5,2)}\pgflineto{\pgfxy(4.6,2.1)}\pgflineto{\pgfxy(4.6,1.9)}\pgflineto{\pgfxy(4.5,2)}\pgfclosepath\pgffillstroke

\pgfmoveto{\pgfxy(5,2.5)}\pgfpatharc{180}{270}{0.5cm}
\pgfstroke
\pgfmoveto{\pgfxy(5,2.5)}\pgflineto{\pgfxy(4.9,2.4)}\pgflineto{\pgfxy(5.1,2.4)}\pgflineto{\pgfxy(5,2.5)}\pgfclosepath\pgffillstroke

\pgfmoveto{\pgfxy(4.5,1)}\pgfpatharc{90}{0}{0.5cm}
\pgfstroke
\pgfmoveto{\pgfxy(4.5,1)}\pgflineto{\pgfxy(4.6,1.1)}\pgflineto{\pgfxy(4.6,0.9)}\pgflineto{\pgfxy(4.5,1)}\pgfclosepath\pgffillstroke

\pgfmoveto{\pgfxy(5,1.5)}\pgfpatharc{180}{270}{0.5cm}
\pgfstroke
\pgfmoveto{\pgfxy(5,1.5)}\pgflineto{\pgfxy(4.9,1.4)}\pgflineto{\pgfxy(5.1,1.4)}\pgflineto{\pgfxy(5,1.5)}\pgfclosepath\pgffillstroke

\pgfputat{\pgfxy(0.2,1)}{\pgfnode{rectangle}{center}{\color{black}\footnotesize $1$}{}{\pgfusepath{}}}
\pgfputat{\pgfxy(0.2,2)}{\pgfnode{rectangle}{center}{\color{black}\footnotesize $2$}{}{\pgfusepath{}}}
\pgfputat{\pgfxy(0.2,3)}{\pgfnode{rectangle}{center}{\color{black}\footnotesize $3$}{}{\pgfusepath{}}}

\pgfputat{\pgfxy(5.8,1)}{\pgfnode{rectangle}{center}{\color{black}\footnotesize $1$}{}{\pgfusepath{}}}
\pgfputat{\pgfxy(5.8,2)}{\pgfnode{rectangle}{center}{\color{black}\footnotesize $2$}{}{\pgfusepath{}}}
\pgfputat{\pgfxy(5.8,3)}{\pgfnode{rectangle}{center}{\color{black}\footnotesize $3$}{}{\pgfusepath{}}}

\pgfputat{\pgfxy(1,0.2)}{\pgfnode{rectangle}{center}{\color{black}\footnotesize $4$}{}{\pgfusepath{}}}
\pgfputat{\pgfxy(2,0.2)}{\pgfnode{rectangle}{center}{\color{black}\footnotesize $5$}{}{\pgfusepath{}}}
\pgfputat{\pgfxy(3,0.2)}{\pgfnode{rectangle}{center}{\color{black}\footnotesize $6$}{}{\pgfusepath{}}}
\pgfputat{\pgfxy(4,0.2)}{\pgfnode{rectangle}{center}{\color{black}\footnotesize $7$}{}{\pgfusepath{}}}
\pgfputat{\pgfxy(5,0.2)}{\pgfnode{rectangle}{center}{\color{black}\footnotesize $8$}{}{\pgfusepath{}}}

\pgfputat{\pgfxy(1,3.8)}{\pgfnode{rectangle}{center}{\color{black}\footnotesize $4$}{}{\pgfusepath{}}}
\pgfputat{\pgfxy(2,3.8)}{\pgfnode{rectangle}{center}{\color{black}\footnotesize $5$}{}{\pgfusepath{}}}
\pgfputat{\pgfxy(3,3.8)}{\pgfnode{rectangle}{center}{\color{black}\footnotesize $6$}{}{\pgfusepath{}}}
\pgfputat{\pgfxy(4,3.8)}{\pgfnode{rectangle}{center}{\color{black}\footnotesize $7$}{}{\pgfusepath{}}}
\pgfputat{\pgfxy(5,3.8)}{\pgfnode{rectangle}{center}{\color{black}\footnotesize $8$}{}{\pgfusepath{}}}
\end{pgfpicture}
\end{tabular}
\caption{Labelled $3\times 5$ diagram (left) with pipe dream construction (right) associated to the diagram $\{(1,2),(1,4), (2,2),(3,1),(3,2),(3,3)\}$ in $\oh_{q}(M_{3,5}(\K))$.\label{FigToricPerm}}
\end{center}
\end{figure}

It was proved in \cite{BellCasteelsLaunois} that $\dim(\ker (M(w))$ is given by the number $r(w)$ of odd cycles in the disjoint cycle decomposition of its associated toric permutation $\tau_w$. Moreover, \cite[Theorem 2.6]{BellLaunoisRogers} shows that all invariant factors of $M(w)$ are powers of $2$. 

Putting all these together, we obtain the following result.
\begin{theorem}\label{thmPIdegCauchonIdeal}
Assume $q$ is a primitive $\ell^{\mbox{th}}$ root of unity with $\ell$ odd.	Let $w$ be a Cauchon diagram for $\oh_q(M_{m,n}(\K))$ and $P_w$ be the corresponding Cauchon ideal. Then
$$\PI(\oh_q(M_{m,n}(\K)) / P_w) = \sqrt{ \ell^{mn-r(w)}},$$
where $r(w)$ is the number of odd cycles in the disjoint cycle decomposition of the toric permutation associated to $w$.
\end{theorem}

Note that we did not need an explicit description of the set $\mathcal{W}'$ of Cauchon diagrams to obtain the above result. For completeness, we note that $\mathcal{W}'$ coincides with the set of $m\times n$ Cauchon-Le diagrams (see \cite[Theorem 4.37]{Thesis} for details).

\section{Irreducible Representations of Maximum Dimension}\label{sectionIrredRepDDA}
We take $1\neq q\in \K^*$ to be arbitrary, unless stated otherwise. Let $A$ be an algebra satisfying Hypothesis~\ref{hyp1} such that $\lambda_{i,j}=q^{m_{i,j}}$ for some skew-symmetric matrix $M=(m_{i,j})_{i,j}\in M_N(\Z)$. Fix $P\in \CSpec_w(A)$, for some Cauchon diagram $w \in \mathcal{W}'$. We make use of the deleting derivations algorithm to construct an irreducible representation of $A/P$ given a ``suitable" irreducible representation of $A'/\psi(P)$. When $A$ is a PI algebra (which holds if $q$ is a root of unity) and $\K$ is algebraically closed, it turns out that any irreducible representation of $A'/\psi(P)$ is ``suitable".

Recall the notation set in Section~\ref{sectionCompletelyPrimeQuotients}: $B:=A/P$ with generators $\bar{X}_i:=X_i+P$ and $B':=A'/\psi(P)$ with generators $t_i:=\bar{T}_i$. Set, again, $\Sigma:=\Sigma_2 \subseteq B'$ to be the (two-sided) Ore set generated by the $t_i$, for $i\in\llbracket 1, N \rrbracket \backslash w$ and then define, recursively for all $j\in \llbracket 2, N \rrbracket$, $\Sigma_{j+1}:=B^{j+1} \cap \Sigma_{j}$, with $\Gamma := \Sigma_{N+1}$ being a (two-sided) Ore set in $B$.

\begin{proposition}\label{propIrredReponBjBk}
Let $A$ be a $\K$-algebra satisfying Hypothesis \ref{hyp1},  $P\in \CSpec_w(A)$ for some $w\in \mathcal{W}'$ and suppose that $(\phi', V)$ is an irreducible representation of $B'$ where, for each $e\in \Sigma$, there exists some $\xi=\xi_e \in \K^*$ and $\ell=\ell_e \in \N_{>1}$ such that $\phi'(e)^{\ell}=\xi \Id_V$. Then
\begin{enumerate}
\item For any $b \in B$, there exists some $b'\in B'$ and $e\in \Sigma$ such that
\[ b = b' e^{-1} \in B'\Sigma^{-1}\]
and $\phi'$ induces an irreducible representation of $B$ on $V$ (of the same dimension), defined by the algebra homomorphism
\begin{align*}
\phi:  B  &{} ~ \longrightarrow ~\End_{\K}(V) \\
b &{}~ \longmapsto ~ \xi^{-1}\phi'(b') \phi'(e)^{\ell -1}.
\end{align*}
\item For each $g\in \Gamma$, there exists some $\xi' \in \K^*$ and $\ell' \in \N_{>1}$ such that $\phi(g)^{\ell'}=\xi' \Id_V$.
\end{enumerate}
\end{proposition}
\begin{proof}
Let $(\phi', V)$ be an irreducible representation of $B'$ satisfying the conditions of the proposition. The condition on $\phi'$ implies that $\phi'(e)^{-1}= \xi_e^{-1}\phi'(e)^{\ell_e-1}$ and induces a representation of $B'\Sigma^{-1}$,
\begin{align*}
\hat{\phi}: ~B' \Sigma^{-1} &{} ~ \longrightarrow ~\End_{\K}(V) \\
b' &{}~ \longmapsto ~\phi'(b') \\
e^{-1} &{} ~\longmapsto ~\xi_e^{-1}\phi'(e)^{\ell_e-1},
\end{align*}
from which we observe that $\hat{\phi}(e)^{-1}=\xi_e^{-1} \hat{\phi}(e)^{\ell_e -1}$, for all $e\in \Sigma$. The inclusion $B'\subseteq B'\Sigma^{-1}$ ensures that $(\hat{\phi}, V)$ is irreducible.

From Proposition \ref{propconfusedalgs3} and Theorem \ref{thmDDAonQuotients} we see that $B'\Sigma^{-1}=B\Gamma^{-1}$, hence each element in $B\Gamma^{-1}$ can be written in terms of elements in $B' \Sigma^{-1}$, and vice versa. This allows us to view $\hat{\phi}$ as an algebra homomorphism $\hat{\phi}: B\Gamma^{-1} \rightarrow  \End_{\K}(V)$, hence $(\hat{\phi}, V)$ defines an irreducible representation of $B\Gamma^{-1}$. 

Every element $b \in B$ may be written as an element $b \cdot 1^{-1} \in B \Gamma^{-1}$, which we write simply as $b\in B\Gamma^{-1}$. By the equality of localisations, $b$ may also be written as an element in $B' \Sigma^{-1}$, that is,
 $b = b' e^{-1} \in B' \Sigma^{-1}$. Restricting $\hat{\phi}$ to $B$ gives the following algebra homomorphism:
\begin{align*}
\phi: B &{}\longrightarrow \End_{\K}(V) \\
b &{} \longmapsto \hat{\phi}\left(b' e^{-1}\right) = \xi^{-1} \phi'(b') \phi'(e)^{\ell-1}.
\end{align*}
This defines a (not necessarily irreducible) representation of $B$ on $V$. To show that this representation is irreducible we note that $\Gamma \subseteq \Sigma$. Therefore, for all $g \in \Gamma$, there exists some $\xi' \in \K^*$ and $\ell' \in \N$ such that $\hat{\phi}(g)^{\ell'}= \xi' \Id_V$, hence
\[\phi(g)^{\ell'} = \hat{\phi}(g)^{\ell'} = \xi' \Id_V. \]
Using the identity $B \Gamma^{-1} = B' \Sigma^{-1}$ we may write any element $b'\in B'$ as $b' = b g^{-1} \in B \Gamma^{-1}$. Thus
\[\phi'(b')=\hat{\phi}(b') = \hat{\phi}\left(bg^{-1}\right) =  \xi'^{-1}\hat{\phi}(b) \hat{\phi}(g)^{\ell'-1}= \phi(\xi'^{-1} b g^{\ell'-1}), \]
which shows that $\phi'(B') \subseteq \phi(B)$. Therefore, since $(\phi', V)$ is irreducible, so too is $(\phi, V)$.
\end{proof}

\begin{remark}
If we take $P=\{0\}$ in Proposition \ref{propIrredReponBjBk} then we get $B'=A'$ and $B=A$ and the statement can be stated analogously given the results and notation of Proposition \ref{propconfusedalgs} and Theorem \ref{thmPIdegs}. This means we can construct an irreducible representation on $A$ provided we have an irreducible representation on $(\phi', V)$ satisfying the condition $\phi'(e)^{\ell_e} = \xi_e \Id_V$ for all elements $e$ of the Ore set in $A'$ generated by $\{T_i\}_{i=1}^N$ and for some $\ell \in \N_{>1}$.
\end{remark}

Restricting the algebra $A$ to the root of unity case and taking $\K$ to be algebraically closed (needed for Schur's lemma used in the proof), we see that any irreducible representation on $B'$ will satisfy the conditions of Proposition \ref{propIrredReponBjBk}.

\begin{theorem}\label{corIrredRepOnB}
Take $\K$ to be an algebraically closed field. Let $A$ be a $\K$-algebra satisfying Hypothesis~\ref{hyp1} and suppose that the $\lambda_{i,j}$ in H\ref{hyp1}.\ref{hyp1.2} are of the form $\lambda_{i,j}=q^{m_{i,j}}$ for some skew-symmetric matrix $M=(m_{i,j})_{i,j}\in M_N(\Z)$, where $q\in \K^*$ is a primitive $\ell^{\text{th}}$ \textbf{root of unity} with $\ell \in \N_{>1}$. Then:
\begin{enumerate}
\item Any irreducible representation $(\phi', V)$ of $B'$ satisfies the conditions of Proposition \ref{propIrredReponBjBk}.
\item Any irreducible representation $(\phi', V)$ of $B'$ induces an irreducible representation $(\phi, V)$ of $B$ of the same dimension.
\end{enumerate} 
\end{theorem}
\begin{proof} \
\begin{enumerate}
\item $A'$ is a prime affine PI algebra (Theorem \ref{thmPIdegs}(iii)), as is $B'$, hence any irreducible representation $(\phi', V)$ of $B'$ is finite dimensional. Suppose $P\in \CSpec_w(A)$, for some $w \in \mathcal{W'}$. Then the quotient algebra, $B'=A'/\psi(P)$, is generated by all $t_i$ such that $i\in \llbracket 1, N \rrbracket \backslash w$, where $t_i:=T_i + \psi(P)$. Hence $t_it_j=q^{m_{i,j}}t_jt_i$, for all $i, j \in  \llbracket 1, N \rrbracket \backslash w$. We deduce that $t_i^{\ell}\in Z(B')$ for all $i\in  \llbracket 1, N \rrbracket \backslash w$, thus, by Schur's Lemma, $\phi'(t_i)^{\ell}=\xi_i \Id_V$ for some $\xi_i \in \K^*$. Since $\Sigma$ is defined to be the multiplicatively closed set generated by $\{t_i \mid i\in  \llbracket 1, N \rrbracket \backslash w\}$, then the conditions of Proposition \ref{propIrredReponBjBk} are satisfied.
\item Follows from Proposition \ref{propIrredReponBjBk}.
\end{enumerate}
\end{proof}

\begin{remark}\label{remIrrepCauchonIdeal}
When $P$ is a Cauchon ideal ($P=P_w$ for some $w \in \mathcal{W}'$) then, by definition, $B'$ becomes a quantum affine space. If $A$ then satisfies the conditions of Theorem \ref{corIrredRepOnB} (i.e.\ we take a quantum nilpotent algebra at a root of unity, $q$) then
 $$ \PI(B') = \prod_{i=1}^{\frac{N-\dim(\ker(M(w)))}{2}} \frac{\ell}{\gcd(h_i, \ell) },$$
 where the $h_i$ are the invariant factors of the matrix $M(w)$ deduced from $M$ by removing rows and columns indexed by $w$
(see Theorem \ref{thmPIdegAgain}). We can construct an irreducible representation of $B'$ of maximal dimension using \cite[Proposition 3.4]{BellLaunoisRogers}. This will, by Theorem \ref{corIrredRepOnB}, induce an irreducible representation of $B$ via the method of Proposition \ref{propIrredReponBjBk}. These results therefore give a method to construct an explicit maximal dimensional irreducible representation of $B$ for a given $w\in\mathcal W'$.
\end{remark}

Proposition \ref{propIrredReponBjBk} and Theorem \ref{corIrredRepOnB} have been stated for the two extreme algebras in the DDA. In practice, we would often deal with one step of the DDA at a time, that is we would deduce an irreducible representation of $A^{(j+1)} /P_{j+1} $ from an irreducible representation of  $A^{(j)} /P_{j} $. 

\subsection{Example: A maximal irreducible representation of $U_q^+(\mathfrak{so}_5)/\langle z' \rangle$}\label{ExampleUq}
To illustrate Remark \ref{remIrrepCauchonIdeal} we will construct an irreducible representation of a quotient of $A:=U_q^+(\mathfrak{so}_5)$ by the ideal $\langle z' \rangle$ (defined later in this section). This example will require applications of the deleting derivations algorithm on both $A$ (Theorem \ref{mainthm}) and $A/\langle z' \rangle$ (Theorem \ref{thmDDAonQuotients}), computing the image of $\langle z' \rangle$ by the canonical embedding (Section \ref{sectionProperties}), and the calculation of the PI degree of $A/ \langle z' \rangle$ (Theorem \ref{thmPIdegAgain}). The full details of these computations are omitted as they are straightforward but lengthy, however they can be found in \cite[Sections 5.3.1 and 7.3.2]{Thesis}.

For this example, we take $q$ to be a primitive $\ell^{\text{th}}$ root of unity with $\ell \notin \{2, 4\}$.

\subsubsection{Defining the algebra}
$U_q^+(\mathfrak{so}_5)$ is the $\C$-algebra generated by two indeterminates $E_1, \, E_2$ subject to the the following relations:
\begin{align*}
E_1^3E_2-(q^2+1+q^{-2})E_1^2 E_2 E_1 + (q^2+1+q^{-2})E_1 E_2 E_1^2-E_2 E_1^3 &{}= 0 \\
E_2^2 E_1 - (q^2+q^{-2})E_2 E_1 E_2 + E_1 E_2^2 &{}= 0.
\end{align*}
There is a PBW basis of $U_q^+(\mathfrak{so}_5)$ formed by monomials $E_1^{k_1}E_4^{k_4}E_3^{k_3}E_2^{k_2}$, where $k_1, k_2, k_3, k_4$ are nonnegative integers and $E_3,\, E_4$ are certain root vectors. This result can be found, for example, in \cite[Section 2.4]{Launois-AutGrpsofEnvelopingAlgs}. The same paper also expresses this algebra as an iterated Ore extension over $\C$ generated by these four indeterminates in the order $E_1, \, E_4,\, E_3,\, E_2$. For easier application of the deleting derivations algorithm, we relabel the indeterminates:
\[ X_1:= E_1, \quad X_2:=E_4, \quad X_3:=E_3, \quad X_4:=E_2. \]
The relations between $X_1,\, X_2,\, X_3,\, X_4$ in  \cite{Launois-AutGrpsofEnvelopingAlgs}  then become
\begin{align*}
X_2X_1&=q^{-2}X_1X_2, \\
X_3X_1&{}=X_1X_3-(q+q^{-1})X_2, \quad &&X_3X_2=q^{-2}X_2X_3, \\
X_4X_1&{}=q^2X_1X_4-q^2X_3, \quad &&X_4X_2=X_2X_4-\frac{q^2-1}{q+q^{-1}}X_3^2, \quad ~~ X_4X_3=q^{-2}X_3X_4.
\end{align*}
This allows us to present $A:=U_q^+(\mathfrak{so}_5)$ as the following iterated Ore extension:
\[U_q^+(\mathfrak{so}_5) = \C[X_1][X_2; \sigma_2][X_3;\sigma_3, \delta_3][X_4;\sigma_4,\delta_4] \]
where, using the notation $A_j:=\C\langle X_1, \ldots, X_j\rangle \subseteq A$, the automorphisms $\sigma_i$ and skew-derivations $\delta_i$ are defined on the generators as:
\begin{align}
\sigma_2: A_1 \longrightarrow A_1; \quad &X_1 \longmapsto q^{-2}X_1, \qquad & &  \nonumber \\
\sigma_3: A_2 \longrightarrow A_2;\quad  &X_1 \longmapsto  X_1 \qquad & \delta_3: A_2 \longrightarrow A_2; \quad &X_1 \longmapsto -(q+q^{-1})X_2 \nonumber \\
&X_2  \longmapsto q^{-2}X_2, \qquad &   &X_2 \longmapsto 0, \nonumber \\
\sigma_4:  A_3 \longrightarrow A_3; \quad &X_1 \longmapsto q^2X_1 \qquad & \delta_4: A_3 \longrightarrow A_3;\quad  &X_1 \longmapsto -q^2 X_3  \nonumber \\
&X_2 \longmapsto  X_2 \qquad & &X_2 \longmapsto -\frac{q^2-1}{q+q^{-1}} X_3^2 \nonumber \\
&X_3 \longmapsto  q^{-2} X_3, \qquad & &X_3 \longmapsto 0.\label{eqnSigmas}
\end{align}

\subsubsection{Verifying that $U_q^+(\mathfrak{s0}_5)$ satisfies Hypothesis \ref{hyp1}}
Routine computations on these maps using \cite[Theorem 2.8 and Lemma 5.3]{Haynal} show that $A$ satisfies all properties of Hypothesis \ref{hyp1}. In particular:
\begin{enumerate}[(H\ref{hyp1}.1)]
\item The $\lambda_{i,j}$'s are as in (\ref{eqnSigmas}).
\item $\lambda_{i,j}=q^{m_{i,j}}$, with the $m_{i,j}$ defining the skew-symmetric matrix
\[M:=\left(\begin{smallmatrix} 0 & 2 & 0 & -2 \\ -2 & 0 & 2 & 0 \\ 0 & -2 & 0 & 2 \\ 2 & 0 & -2 & 0 \end{smallmatrix} \right) \in M_4(\Z).\]
\item $(\sigma_3, \delta_3)$ is $q^2$-skew and $(\sigma_4, \delta_4)$ is $q^4$-skew. So we set $q_3=q^2$ and $q_4=q^4$.
\item The higher $q_j$-skew $\sigma_j$-derivations are defined to be:
\begin{align}\label{eqnd_(j,n)}
d_{j,n}(X_l) &{}= \begin{cases} X_l & n=0; \\
								 \delta_j(X_l) &  n=1; \\
								 0 & n>1,
					\end{cases}
\end{align}
for $j\in \{3,4\}$ and $l \in \llbracket 1, j-1 \rrbracket$.
\end{enumerate} 
Since there are two nonzero derivations we apply Theorem \ref{mainthm} twice to obtain $A'=\C_{q^M}[T_1, T_2, T_3, T_4]$, where
\begin{align}\label{eqnX_j^(3)asX_j}
T_4&{}:=X_4, &T_3&{}:=X_3,  & T_2&{}:=X_2-\frac{q^4}{(q^2+1)(q+q^{-1})} X_3^2 X_4^{-1},  & T_1&{}:= X_1-\frac{q^2(q+q^{-1})}{q^2-1}X_2 X_3^{-1}.
\end{align}

\subsubsection{Finding the PI degree of  $U_q^+(\mathfrak{s0}_5)/\langle z' \rangle$}
Consider the ideal $\langle z' \rangle \lhd A$, generated by the central element (see \cite[Section 2.4]{Launois-AutGrpsofEnvelopingAlgs})
\[ z' := -(q^2-q^{-2})(q+q^{-1})X_2X_4 + q^2(q^2-1)X_3^2. \]
Using localisation theory, Proposition \ref{propconfusedalgs}(\ref{confusedalg3}), and repeated application of Proposition \ref{lemBiHom} and Lemma \ref{lemBiIncHomP0}, one can show that $\psi^{-1}(\langle T_2 \rangle) = \langle z' \rangle$. In particular this proves that $\langle z' \rangle \in \CSpec_{\{2\}}(A)$ and therefore $\{2\} \in \mathcal{W}'$. In fact, since $A'/\langle T_2 \rangle = \C_{q^{M'}}[t_1, t_3, t_4]$, where $t_i := T_i + \langle T_2 \rangle$ for all $i\in \{1, 3, 4\}$, then $\langle z' \rangle$ is a Cauchon ideal.

By Theorem \ref{thmDDAonQuotients}, 
\[\PI(U_q^+(\mathfrak{s0}_5)/\langle z' \rangle) = \PI(A'/\langle T_2 \rangle) = \PI(\C_{q^{M'}}[t_1, t_3, t_4])\]
where $M'=\left(\begin{smallmatrix} 0 & 0 & -2 \\ 0 & 0 & 2 \\ 2 & -2 & 0 \end{smallmatrix} \right)$ is obtained from $M$ by deleting the second row and second column. It is easily verified that the skew normal form of $M'$ is $S=\left(\begin{smallmatrix} 0 & 2 & 0 \\ -2 & 0 & 0 \\ 0 & 0 & 0 \end{smallmatrix}\right)$, hence $M'$ has a kernel of dimension $1$ and one pair of invariant factors: $h_1 = 2$. Applying Theorem \ref{thmPIdegAgain} with $P_{\{2\}}=\langle z' \rangle$ gives:
\[ \text{PI-deg}(A/\langle z' \rangle) = \prod_{i=1}^{\frac{3-1}{2}} \frac{\ell}{\gcd(h_i, \ell)} = \frac{\ell}{\gcd(2, \ell)}=\begin{cases} \ell & \ell \text{ is odd};\\ \ell/2 & \ell>4 \text{ is even}. \end{cases}\]

\subsubsection{Constructing an irreducible representation of $U_q^+(\mathfrak{so}_5)/\langle z' \rangle$ of maximum dimension}\label{sectionEqUq}
Recall that $B':=\C_{q^{M'}}[t_1, t_3, t_4], ~ B:=U_q^+(\mathfrak{so}_5)/\langle z' \rangle$ and let $\Sigma \subseteq B'$ be the multiplicatively closed set generated by $\{t_1, t_3, t_4\}$. The quantum affine space associated to $S$ is $D:=\C_{q^S}[x_1, y_1, z_1]$. By \cite[Proposition 3.4(ii)]{BellLaunoisRogers} there is an $\ell$-dimensional $\C$-vector space $V$, $\lambda, \xi\in\C^*$ and an algebra homomorphism $\varphi: D \rightarrow \End_{\C}(V)$ whose image on the generators of $D$, upon fixing a basis $v_1, \ldots, v_{\ell}\in V$, may be presented as matrices in the following way:
\[\varphi(x_1) = \left(\begin{smallmatrix} \lambda &  & &  & \\
												 & \lambda q^2 & & & \\
												  &  & \lambda q^4 &  & \\
												  & & & \ddots &  \\
												  &  & &  & \lambda q^{(\ell-1)2}
							\end{smallmatrix} \right), \qquad  \varphi(y_1) = \left(\begin{smallmatrix} 0 & 0 &  \ldots & 0 & 1 \\
											1 & 0 &  \ldots & 0 & 0 \\
											0& 1 &  \ldots & 0 & 0 \\
											\vdots &  & \ddots  &\vdots & \vdots \\
											0 & \ldots & \ldots  & 1 & 0 \end{smallmatrix}\right), \qquad \varphi(z_1)= \xi \Id_V. \]
The pair $(\varphi, V)$ defines an irreducible representation of $D$.

To define an irreducible representation of $\C_{q^{M'}}[t_1, t_3, t_4]$ we apply \cite[Proposition 3.4(iii)]{BellLaunoisRogers} using $E^{-1}=\left(\begin{smallmatrix} 1 & 0 & 1 \\ -1 & 0 & 0 \\ 0 & -1 & 0  \end{smallmatrix}\right)$, where $E\in M_3(\Z)$ is the invertible matrix satisfying $EM'E^T=S$. The resulting algebra homomorphism $\phi': \C_{q^{M'}}[t_1, t_3, t_4] \rightarrow \End_{\C}(V)$, is defined on $t_1, \, t_3, \, t_4$ as
\begin{align}\label{eqnPhiont}
\phi'(t_1) &{}= \varphi(x_1^1 y_1^0 z_1^1) = \varphi(x_1) \varphi(z_1) = \xi \varphi(x_1), \nonumber \\
\phi'(t_3) &{}= \varphi(x_1^{-1} y_1^0 z_1^0) = \varphi(x_1)^{-1} = \lambda^{-\ell} \varphi(x_1^{\ell-1}), \nonumber \\
\phi'(t_4) &{}= \varphi(x_1^0 y_1^{-1} z_1^0) = \varphi(y_1)^{-1} = \varphi(y_1^{\ell-1}),
\end{align} 
and it defines an irreducible representation of $\C_{q^{M'}}[t_1, t_3, t_4]$ on $V$. Using the definition of $\varphi$, we see that 
\begin{align*}
\phi'(t_1)^{\ell} &{}= \xi^{\ell} \varphi(x_1)^{\ell} = (\xi\lambda)^{\ell} \Id_V, \nonumber\\
\phi'(t_3)^{\ell} &{} = \lambda^{-\ell^2} \varphi(x_1)^{\ell^2-\ell}= \lambda^{-\ell^2} \lambda^{\ell^2-\ell} \Id_V = \lambda^{-\ell} \Id_V, \nonumber\\
\phi'(t_4)^{\ell} &{}= \varphi(y_1)^{\ell^2-\ell}=\Id_V,
\end{align*}
hence the conditions of Proposition \ref{propIrredReponBjBk} are satisfied. Applying this proposition allows us to define an irreducible representation of $B$ once we know how to write the generators $\bar{X}_1, \ldots, \bar{X}_4 \in B$ in terms of the generators $t_1, t_3, t_4 \in B'$. Using (\ref{eqnX_j^(3)asX_j}) we see that
\begin{align*}
t_4:=\bar{X}_4, \quad t_3=\bar{X}_3, \quad t_2=\bar{X}_2-\frac{q^4}{(q^2+1)(q+q^{-1})} \bar{X}_3^2 \bar{X}_4^{-1}, \quad t_1= \bar{X}_1 - \frac{q^2(q+q^{-1})}{q^2-1} \bar{X}_2 \bar{X}_3^{-1}.
\end{align*}
Rearranging these identities, and noting that $t_2=0$, allows us to write
\[ \bar{X}_4=t_4, \quad \bar{X}_3=t_3, \quad \bar{X}_2=\frac{q^4}{(q^2+1)(q+q^{-1})} t_3^2 t_4^{-1}, \quad \bar{X}_1=t_1+\frac{q^4}{q^4-1} t_3 t_4^{-1}.\]
Applying Proposition \ref{propIrredReponBjBk} and using (\ref{eqnPhiont}), we deduce that there is an algebra homomorphism $\phi: B \rightarrow \End_{\C}(V)$ defined on generators as
\begin{align*}
\phi(\bar{X}_4)&{}=\phi'(t_4) = \varphi(y_1)^{-1}, & \phi(\bar{X}_3)&{}=\phi'(t_3) = \varphi(x_1)^{-1}, \\
\phi(\bar{X}_2)&{}=\phi'\left(\frac{q^4}{(q^2+1)(q+q^{-1})} t_3^2 t_4^{-1}\right) & \phi(\bar{X}_1)&{}=\phi'\left(t_1+\frac{q^4}{q^4-1} t_3 t_4^{-1}\right) \\
&{}= \frac{q^4}{(q^2+1)(q+q^{-1})} \varphi(x_1)^{-2} \varphi(y_1), & &{}= \xi \varphi(x_1)+ \frac{q^4}{q^4-1} \varphi(x_1)^{-1} \varphi(y_1).\\
\end{align*}
which defines an irreducible representation $(\phi, V)$ of $B$. Substituting in the matrices for $\varphi(x_1)$ and $\varphi(y_1)$ and taking, for example, $\ell=5$ gives the following explicit form of this irreducible representation:
\begin{align*}
\phi(\bar{X}_4)&{}= \left( \begin {matrix} 0&1&0&0&0\\ \noalign{\medskip}0&0&1&0&0
\\ \noalign{\medskip}0&0&0&1&0\\ \noalign{\medskip}0&0&0&0&1
\\ \noalign{\medskip}1&0&0&0&0\end {matrix} \right), \quad & \phi(\bar{X}_3)&{}=  \left( \begin {matrix} {\lambda}^{-1}&0&0&0&0
\\ \noalign{\medskip}0&\lambda^{-1} q^3&0&0&0
\\ \noalign{\medskip}0&0&\lambda^{-1}q&0&0
\\ \noalign{\medskip}0&0&0&\lambda^{-1} q^4&0
\\ \noalign{\medskip}0&0&0&0&\lambda^{-1} q^2\end {matrix}
 \right),
\end{align*}
\begin{align*}
   \phi(\bar{X}_2) &{}= \left( \begin {matrix} \noalign{\medskip}0&0&0&0&\frac{1}{(q^2+1)^2 \lambda^2}\\ \noalign{\medskip}\frac{q}{(q^2+1)^2 \lambda^2}&0&0&0&0
\\ \noalign{\medskip}0&\frac{q^2}{(q^2+1)^2 \lambda^2}&0&0&0\\ \noalign{\medskip}0&0&\frac{q^3}{(q^2+1)^2 \lambda^2}&0&0  \\ \noalign{\medskip}0&0&0&\frac{q^4}{(q^2+1)^2 \lambda^2}&0\end{matrix} \right), \\
\phi(\bar{X}_1)&{}=\left( \begin {matrix}\xi \lambda &0&0&0
& \frac{q^4}{(q^{4}-1)\lambda} \\ \noalign{\medskip}\frac{q^2}{(q^{4}-1)\lambda}&\xi \lambda q^2&0&0&0
\\ \noalign{\medskip}0&\frac{1}{(q^{4}-1)\lambda}&\xi \lambda q^4&0&0
\\ \noalign{\medskip}0&0&\frac{q^3}{(q^{4}-1)\lambda}&\xi \lambda q &0 \\ \noalign{\medskip}0&0&0&\frac{q}{(q^{4}-1)\lambda}&\xi \lambda q^3
\end {matrix} \right).
\end{align*}

\bibliographystyle{alpha}

\begin{thebibliography}{CPWZ16}

  \bibitem[BCL12]{BellCasteelsLaunois}
J.~P.~Bell, K.~Casteels, and S.~Launois.
\newblock Enumeration of {$\mathcal{H}$}-strata in quantum matrices with
  respect to dimension.
\newblock {\em Journal of Combinatorial Theory. Series A}, 119(1):83--98, 2012.
	
\bibitem[BLR22]{BellLaunoisRogers}
J.~Bell, S.~Launois, and A.~Rogers.
\newblock {PI} degree and irreducible representations of quantum determinantal
  rings and their associated quantum schubert varieties.
\newblock arXiv.2212.03799v2
																						
\bibitem[BG02]{BrownGoodearl}
K.~A. Brown and K.~R. Goodearl.
\newblock {\em Lectures on algebraic quantum groups}.
\newblock Advanced Courses in Mathematics. CRM Barcelona. Birkh\"auser Verlag,
  Basel, 2002.
			 

\bibitem[BY17]{BrownYakimov}
K.~A. {Brown} and M.~T. {Yakimov}.
\newblock {Azumaya loci and discriminant ideals of PI algebras}.
\newblock {\em Advances in Mathematics}, 340: 1219--1255, 2018.

\bibitem[Cas14]{Casteels}
K.~Casteels.
\newblock Quantum matrices by paths.
\newblock {\em Algebra \& Number Theory}, 8(8):1857--1912, 2014.

\bibitem[Cau03a]{Cauchon}
G.~Cauchon.
\newblock Effacement des d\'erivations et spectres premiers des alg\`ebres
  quantiques.
\newblock {\em Journal of Algebra}, 260(2):476--518, 2003.
		
\bibitem[Cau03b]{Cauchon-SpectrePremiers}
G.~Cauchon.
\newblock Spectre premier de {$O_q(M_n(k))$}: image canonique et s\'{e}paration
  normale.
\newblock {\em Journal of Algebra}, 260(2):519--569, 2003.
																				  

\bibitem[CPWZ15]{CekenPalmieriWangZhang2}
S.~Ceken, J.~H. Palmieri, Y.-H. Wang, and J.~J. Zhang.
\newblock The discriminant controls automorphism groups of noncommutative
  algebras.
\newblock {\em Advances in Mathematics}, 269:551--584, 2015.
																							  
\bibitem[CPWZ16]{CekenPalmieriWangZhang1}
S.~Ceken, J.~H. Palmieri, Y.-H. Wang, and J.~J. Zhang.
\newblock The discriminant criterion and automorphism groups of quantized
  algebras.
\newblock {\em Advances in Mathematics}, 286:754--801, 2016.
						
\bibitem[CP93]{DeConciniProcesi}
C.~De Concini and C.~Procesi.
\newblock Quantum groups.
\newblock In {\em {$D$}-modules, representation theory, and quantum groups
  ({V}enice, 1992)}, volume 1565 of {\em Lecture Notes in Math.}, pages
  31--140. Springer, Berlin, 1993.
  
\bibitem[GY14]{GeigerYakimov}
J.~Geiger and M.~Yakimov.
\newblock Quantum Schubert cells via representation theory and ring theory.
\newblock {\em Michigan Mathematical Journal}, 63: 125--157, 2014. 
  
\bibitem[GLL19]{GoodearlLaunoisLenagan-Tauvel}
K.~R. {Goodearl}, S.~{Launois}, and T.~H.~{Lenagan}.
\newblock {Tauvel's height formula for quantum nilpotent algebras}.
\newblock {\em  Communications in Algebra}, 47(10): 4194--4209, 2019. 
																			  
			
\bibitem[GL98]{GoodearlLetzter}
K.~R. Goodearl and E.~S. Letzter.
\newblock Prime and primitive spectra of multiparameter quantum affine spaces.
\newblock In {\em Trends in ring theory ({M}iskolc, 1996)}, volume~22 of {\em
  CMS Conf. Proc.}, pages 39--58. American Mathematical Society, Providence,
  RI, 1998.
		

\bibitem[GW04]{GoodearlWarfield}
K.~R. Goodearl and R.~B. Warfield, Jr.
\newblock {\em An introduction to noncommutative {N}oetherian rings}, volume~61
  of {\em London Mathematical Society Student Texts}.
\newblock Cambridge University Press, Cambridge, second edition, 2004.

\bibitem[GY16]{GoodearlYakimov}
K.~R. Goodearl and M.~T. Yakimov.
\newblock From quantum Ore extensions to quantum tori via noncommutative UFDs.
\newblock {\em Advances in Mathematics},  300: 672--716,  2016. 

\bibitem[Hay08]{Haynal}
H.~Haynal.
\newblock P{I} degree parity in {$q$}-skew polynomial rings.
\newblock {\em Journal of Algebra}, 319(10):4199--4221, 2008.

\bibitem[KLR04]{KellyLenaganRigal}
A.~C. Kelly, T.~H. Lenagan, and L.~Rigal.
\newblock {R}ing theoretic properties of quantum {G}rassmannians.
\newblock {\em Journal of Algebra and its Applications}, 3(1):9--30, 2004.

\bibitem[Lau06]{Launois-AutGrpsofEnvelopingAlgs}
S.~Launois.
\newblock {On the automorphism groups of $q$-enveloping algebras of nilpotent Lie algebras}.
\newblock In Proceedings {\em From Lie Algebras to Quantum Groups}, 2006.


\bibitem[LL17]{LaunoisLecoutre}
S.~Launois, C.~Lecoutre.
\newblock Poisson deleting derivations algorithm and {P}oisson spectrum
\newblock {\em Communications in Algebra},
  45(3):1294--1313, 2017.

\bibitem[LLN19]{LaunoisLenaganNolan}
S.~Launois, T.~H. Lenagan, and B.~Nolan.
\newblock Total positivity is a quantum phenomenon: the grassmannian case.
\newblock To appear in Memoirs of the American Mathematical Society,  2019.

\bibitem[LLR06]{LaunoisLenaganRigal2}
S.~Launois, T.~H. Lenagan, and L.~Rigal.
\newblock Quantum unique factorisation domains.
\newblock {\em Journal of the London Mathematical Society. Second Series},
  74(2):321--340, 2006.
  


\bibitem[LR06]{LenaganRigal-Straightening}
T.~H. Lenagan and L.~Rigal.
\newblock Quantum graded algebras with a straightening law and the
  {AS}-{C}ohen-{M}acaulay property for quantum determinantal rings and quantum
  {G}rassmannians.
\newblock {\em Journal of Algebra}, 301(2):670--702, 2006.

\bibitem[LR08]{LenaganRigal-SchubertVar}
T.~H. Lenagan and L.~Rigal.
\newblock Quantum analogues of {S}chubert varieties in the {G}rassmannian.
\newblock {\em Glasgow Mathematical Journal}, 50(1):55--70, 2008.

\bibitem[LM11]{LeroyMatczuk}
A.~Leroy and J.~Matczuk.
\newblock On {$q$}-skew iterated {O}re extensions satisfying a polynomial
  identity.
\newblock {\em Journal of Algebra and its Applications}, 10(4):771--781, 2011.

\bibitem[MR01]{McConnellRobson}
J.~C. McConnell and J.~C. Robson.
\newblock {\em Noncommutative {N}oetherian rings}, volume~30 of {\em Graduate
  Studies in Mathematics}.
\newblock American Mathematical Society, Providence, RI, revised edition, 2001.
\newblock With the cooperation of L. W. Small.

\bibitem[New72]{Newman}
M.~Newman.
\newblock {\em Integral matrices}.
\newblock Academic Press, New York-London, 1972.
\newblock Pure and Applied Mathematics, Vol. 45.

\bibitem[{Pos}06]{Postnikov}
A.~Postnikov.
\newblock {Total positivity, Grassmannians, and networks}.
\newblock {\em arXiv Mathematics e-prints}, page math/0609764, Sep 2006.

\bibitem[Rog19]{Thesis}
A.~Rogers.
\newblock {\em Irreducible Representations of Quantum Nilpotent Algebras at
  Roots of Unity, and Their Completely Prime Quotients}.
\newblock PhD thesis, University of Kent, 2019.



\end{thebibliography}

\end{document}